\documentclass[reqno]{amsart}
\usepackage{amsmath,amsthm,amsfonts,amssymb,amscd,verbatim,multicol}
\usepackage[arrow,matrix,cmtip,curve]{xy}
\usepackage{lscape}
\usepackage{mathrsfs}  
\usepackage{fullpage}
\usepackage{stmaryrd}
\usepackage{xcolor}
\allowdisplaybreaks

\usepackage{xparse}
\usepackage{enumitem}

\usepackage[normalem]{ulem}

\numberwithin{equation}{section}

\theoremstyle{definition}
\newtheorem{theorem}{Theorem}[section]
\newtheorem{theorem*}{Theorem}

\newtheorem{corollary}[theorem]{Corollary} 
\newtheorem{definition}[theorem]{Definition} 
\newtheorem{example}[theorem]{Example}
\newtheorem{lemma}[theorem]{Lemma}
\newtheorem{proposition}[theorem]{Proposition}

\newtheorem{remark}[theorem]{Remark}

\DeclareMathOperator\ad{ad}

\DeclareMathOperator\Der{Der}

\DeclareMathOperator\End{End}

\DeclareMathOperator\ev{ev}
\DeclareMathOperator\gr{gr}

\DeclareMathOperator\id{id}

\DeclareMathOperator\op{op}

\renewcommand\int{\mathrm{int}}

\newcommand\tensor{\otimes}
\newcommand\stensor{\,\widehat\otimes\,}
\newcommand\kk{\Bbbk}

\newcommand\NN{\mathbb N}

\newcommand\ZZ{\mathbb Z}

\renewcommand\d{\mathrm{d}}

\newcommand\ep{\varepsilon}
\renewcommand\phi{\varphi}
\newcommand\del{\nabla}

\renewcommand\ev{\mathrm{ev}}

\begin{document}

\title{Universal Enveloping Algebras of Poisson Superalgebras}

\author[Lamkin]{Thomas Lamkin}
\address{UCSD, Department of Mathematics, 9500 Gilman Dr, La Jolla, CA 92093} 
\email{tlamkin@ucsd.edu}

\subjclass[2010]{17B63, 17B60, 17B35, 16T05}
\keywords{Poisson superalgebra, enveloping algebra, Lie-Rinehart superalgebra, PBW Theorem}
\begin{abstract}
In this paper, we define and study the universal enveloping algebra of a Poisson superalgebra. In particular, a new PBW Theorem for Lie-Rinehart superalgebras is proved leading to a PBW Theorem for Poisson superalgebras, we show the universal enveloping algebra of a Poisson Hopf superalgebra (resp. Poisson-Ore extension) is a Hopf superalgebra (resp. iterated Ore extension), and we study the universal enveloping algebra for interesting classes of Poisson superalgebras such as Poisson symplectic superalgebras.
\end{abstract}

\maketitle

\setcounter{section}{-1}
\section{Introduction}

The notion of a Poisson algebra arises naturally in the study of Hamiltonian mechanics and has since found use in numerous areas such as Poisson and symplectic geometry, as well as in the study of quantum groups. Due to the development of supersymmetry theories, one has also witnessed increased interest in ``super" objects such as Lie superalgebras, supermanifolds, and—most importantly for this paper—Poisson superalgebras. Naturally, one may be interested in studying the representation theory of algebraic super objects. In the non-super case, one method of studying the representation theory of Poisson algebras is to study its \emph{universal enveloping algebra}, a construction introduced in \cite{Oh_UEA} by Oh. Since then, universal enveloping algebras have been studied in a variety of contexts, such as for Poisson Hopf algebras \cite{LWZ_PoissonHopf, Oh_PoissonHopf} and Poisson-Ore extensions \cite{LWZ_PoissonOre}, and partial Poincaré–Birkhoff–Witt (PBW) theorems have been obtained \cite{LOV_PBW, LWZ_PoissonHopf, OPS_PBW}. We extend several such results to the case of Poisson superalgebras, and we prove the PBW Theorem holds for all Poisson (super)algebras over a field.

Our approach to prove the PBW Theorem uses the relation between Lie-Rinehart superalgebras and Poisson superalgebras. The former is a super generalization of Lie-Rinehart algebras which were first given comprehensive treatment by Rinehart in \cite{Rinehart_PBW}, and can be viewed as an algebraic analogue of the more geometric notion of Lie algebroids. In his paper, Rinehart defined the universal enveloping algebra of a Lie-Rinehart algebra $(A,L)$ and proved a PBW theorem for Lie-Rinehart algebras in the case $L$ is projective as an $A$-module (see \cite{Rinehart_PBW} and Section \ref{sec.LR_PBW} of this paper). Later in \cite{huebschmann_Poisson/LR}, it was shown that Poisson algebras can be viewed as Lie-Rinehart algebras in such a way that the enveloping algebra in the Poisson sense is isomorphic to the enveloping algebra in the Lie-Rinehart sense. The authors of \cite{LWZ_PoissonHopf} then used this perspective to prove a partial PBW theorem for Poisson algebras, with the aforementioned projectivity condition translating in the Poisson case to requiring the K\"{a}hler differentials $\Omega_R$ be projective over the Poisson algebra $R$. While the PBW Theorem was known to hold for other Poisson algebras, see e.g. \cite{LOV_PBW}, the general problem was still open. We show the PBW theorem does indeed hold for all Poisson (super)algebras over a field by proving the PBW theorem holds for Lie-Rinehart superalgebras $(A,L)$ over a commutative ring $S$ where $A$ and $L$ are free (super)modules over $S$.

This paper is organized as follows.

In Section \ref{sec.prelims}, we recall the necessary background in supermathematics. In Section \ref{sec.Poisson_superalgs}, we recall the definition of Poisson superalgebras and morphisms between them, and introduce some interesting examples of Poisson superalgebras. Next, in Section \ref{sec.UEA} we define and construct the universal enveloping algebra of a Poisson superalgebra, prove standard results such as the correspondence between Poisson modules over a Poisson superalgebra and modules over its universal enveloping algebra, and study the universal enveloping algebra of the Poisson superalgebras introduced in Section \ref{sec.Poisson_superalgs}. Section \ref{sec.LR_PBW} is dedicated to proving a new PBW Theorem for Lie-Rinehart superalgebras:
\begin{theorem*}[Theorem \ref{thm.LR_PBW}]
Suppose $A$ and $L$ are free supermodules over the base commutative ring $S$. Then the canonical $A$-superalgebra homomorphism $S_A(L)\rightarrow \gr(V(A,L))$ is an isomorphism.
\end{theorem*}
In Section \ref{sec.Poisson_PBW}, we follow \cite{LWZ_PoissonHopf} and use the results of Section \ref{sec.LR_PBW} to prove the PBW Theorem for Poisson superalgebras. In Section \ref{sec.Poisson_Hopf}, we study the universal enveloping algebra of Poisson Hopf superalgebras. The following result is a super generalization of \cite[Theorem 10]{Oh_PoissonHopf}.
\begin{theorem*}[Theorem \ref{thm.Hopf_UEA}]
If $(R,\eta,\del,\ep,\Delta,S)$ is a Poisson Hopf superalgebra, then the universal enveloping algebra $U(R)$ is a Hopf superalgebra.
\end{theorem*}
Finally, in Section \ref{sec.Poisson_Ore}, we define Poisson-Ore extensions for Poisson superalgebras by an even indeterminate and show their universal enveloping algebras are iterated Ore extensions, generalizing \cite[Theorem 0.1]{LWZ_PoissonOre}.

\subsection*{Acknowledgements}
The author would like to thank Miami University and the USS program for funding this project. The author would also like to thank Jason Gaddis for suggesting the project idea, as well as for many helpful discussions and advice throughout.

\section{Background on Supermathematics}\label{sec.prelims}
Throughout this paper we work over a field $\kk$ of characteristic 0, except in Section \ref{sec.LR_PBW} where we work over a commutative ring $S$ of characteristic 0. All (super)algebras, (super) tensor products, etc., are assumed to be taken over the base field (or ring) unless specified otherwise. We recall the basic concepts of supermathematics, referring the reader to \cite{Manin_Supermath} for further details.

\begin{definition}\label{defn.superring}
A \emph{superring} is a $\ZZ_2$-graded ring $R=R_0\oplus R_1$ such that $R_iR_j\subseteq R_{i+j}$, where the indices are taken mod 2. Elements of $R_0$ are called \emph{even}, elements of $R_1$ are called \emph{odd}, and an element which is either even or odd is called \emph{homogeneous}. We denote the \emph{parity} or \emph{degree} of a homogeneous element $r$ by $|r|$.

A superring is \emph{supercommutative} if
\begin{equation*}
rs=(-1)^{|r||s|}sr
\end{equation*}
for all homogeneous $r,s\in R$. 
\end{definition}

In the sequel, we implicitly assume indices are taken mod 2 in relations such as $R_iR_j\subseteq R_{i+j}$. Note that any ring $R$ can be made into a superring via the trivial grading $R=R\oplus 0$.

\begin{definition}\label{defn.supermodule}
A \emph{(left) supermodule} over a superring $R=R_0\oplus R_1$ is a $\ZZ_2$-graded (left) module $M=M_0\oplus M_1$ such that $R_iM_j\subseteq M_{i+j}$. Right supermodules are defined similarly.
\end{definition}

For supermodules, and superobjects in general, even, odd, and homogeneous elements are defined as for superrings and $|\cdot|$ will be used to denote parity. Throughout this paper, we will adopt the convention that elements appearing in any formula or expression are homogeneous unless stated otherwise.

Note that if a superring $R$ is supercommutative, then any left $R$-supermodule $M$ can naturally be considered a right $R$-supermodule via the action $m*r=(-1)^{|r||m|}r\cdot m$. Thus, we shall assume all supermodules are left supermodules unless stated otherwise. We will say a supermodule $M=M_0\oplus M_1$ is a \emph{free} supermodule if $M$ has a basis consisting of homogeneous elements. For example, a super vector space is always a free supermodule over its base field.

\begin{definition}\label{defn.supermodule_hom}
A \emph{homomorphism} from an $R$-supermodule $M$ to an $R$-supermodule $N$ is an $R$-linear map $\phi:M\rightarrow N$. If $\phi$ preserves the grading, we say $\phi$ is an \emph{even} homomorphism whereas if $\phi$ reserves the grading, we say $\phi$ is an \emph{odd} homomorphism. A homomorphism is \textit{homogeneous} if it is either even or odd.
\end{definition}

\begin{definition}\label{defn.superalg}
A \emph{superalgebra} over a supercommutative superring $R=R_0\oplus R_1$ is an $R$-supermodule $A=A_0\oplus A_1$ together with an $R$-bilinear multiplication $A\times A\rightarrow A$ that respects the grading, is associative, and admits a unit. That is,
\begin{equation*}
    r\cdot(ab)=(r\cdot a)b=(-1)^{|r||a|}a(r\cdot b)
\end{equation*}
for all $r\in R$ and $a,b\in A$, and $A_iA_j\subseteq A_{i+j}$. A \textit{homomorphism of superalgebras} is a homomorphism of algebras that is also a homomorphism of supermodules. A superalgebra is \emph{supercommutative} if it is supercommutative as a superring.
\end{definition}

Another way to interpret the supercommutativity condition for superalgebras is as follows. Given a superalgebra $A$, define the \emph{opposite} superalgebra to be the superalgebra $A^{\op}$ whose underlying supermodule is the same as $A$, but with multiplication defined by
\begin{equation*}
    a\circ b:=(-1)^{|a||b|}ba.
\end{equation*}
Then $A$ is supercommutative if and only if $A=A^{\op}$.

\begin{definition}\label{defn.superder}
Let $A,B$ be superalgebras and let $M$ be a $B$-supermodule. If $\sigma:A\rightarrow B$ is a homogeneous superalgebra homomorphism and $d:A\rightarrow M$ is a homogeneous $A$-linear map, then we say $d$ is a \textit{homogeneous $\sigma$-superderivation} (of degree $|d|$) if
\begin{equation*}
    d(ab)=(-1)^{|a||d|}\sigma(a)d(b)+(-1)^{|\sigma|(|d|+|a|)}d(a)\circ\sigma(b)=(-1)^{|a||d|}\sigma(a)d(b)+(-1)^{|b|(|d|+|a|)}\sigma(b)d(a)
\end{equation*}
for all homogeneous $a,b\in A$. A \textit{$\sigma$-superderivation} is the sum of an even $\sigma$-superderivation and an odd $\sigma$-superderivation; a \textit{superderivation} is an $\id_A$-superderivation.
\end{definition}

\begin{definition}\label{defn.super_tensor_prod}
If $M$ is a right supermodule and $N$ is a left supermodule, then the \emph{super tensor product} of $M$ and $N$, denoted $M\stensor N$, is the super abelian group with grading
\begin{align*}
    &(M\stensor N)_0=(M_0\tensor N_0)\oplus(M_1\tensor N_1),\\
    &(M\stensor N)_1=(M_0\tensor N_1)\oplus(M_1\tensor N_0).
\end{align*}
\end{definition}

Note that if $A,B$ are superalgebras over a supercommutative superring $R$, then $A\stensor B$ is a superalgebra over $R$ via the multiplication defined on homogeneous elements by
\begin{equation*}
(a\stensor b)(c\stensor d)=(-1)^{|b||c|}(ac\stensor bd).
\end{equation*}

\begin{definition}
Let $M$ be a module over a supercommutative superring $R$. Then the \emph{tensor superalgebra} of $M$ is
\begin{equation*}
    T_R(M):=\bigoplus_{k\leq 0}M^{\stensor k}.
\end{equation*}
The \emph{supersymmetric superalgebra} of $M$ is the quotient
\begin{equation*}
    S_R(M):=T_R(M)/(x\stensor y-(-1)^{|x||y|}y\stensor x).
\end{equation*}
If $R$ is the base field $\kk$ (or base ring $S$ in Section \ref{sec.LR_PBW}), then we will denote the tensor superalgebra and supersymmetric superalgebra as simply $T(M),S(M)$, respectively.
\end{definition}

The tensor superalgebra and supersymmetric superalgebras both have universal properties analogous to those of their non-super counterparts; both can be found in \cite[Chapter 3, \S 2.5]{Manin_Supermath}.

\begin{definition}
Given superalgebra homomorphisms $f:A\rightarrow A'$ and $g:B\rightarrow B'$, the \textit{super tensor product of f and g} is the superalgebra homomorphism
\begin{align*}
    f\stensor g:\hspace{6pt}&A\stensor B\rightarrow A'\stensor B'\\
    a\stensor b&\mapsto (-1)^{|a||g|}f(a)\stensor g(b).
\end{align*}
\end{definition}

\section{Poisson Superalgebras}\label{sec.Poisson_superalgs}
In this section we will recall the definition of a Poisson superalgebra, and provide useful definitions and examples that will be used in subsequent sections.

\begin{definition}\label{defn.Lie_superalgebra}
A super vector space $L=L_0\oplus L_1$ together with a bilinear bracket $\{\cdot,\cdot\}:L\times L\rightarrow L$ is a \textit{Lie superalgebra} if
$\{L_i,L_j\}\subseteq L_{i+j}$, and for homogeneous $x,y\in L$,
\begin{enumerate}
    \item $\{x,y\}=-(-1)^{|x||y|}\{y,x\}$ (Super skew-symmetry),
    \item $\{x,\{y,z\}\}+(-1)^{|x|(|y|+|z|)}\{y,\{z,x\}\}+(-1)^{|z|(|x|+|y|)}\{z,\{x,y\}\}=0$ (Super Jacobi identity).
\end{enumerate}
In particular, if $A$ is an associative superalgebra then the pair $(A,[\cdot,\cdot]_{\gr})$, where $[\cdot,\cdot]_{\gr}$ is the supercommutator bracket, is a Lie superalgebra which we denote by $A_L$.
\end{definition}

\begin{definition}\label{defn.Poisson_superalgebra}
A \emph{Poisson superalgebra} is a supercommutative superalgebra $R$ with a bracket $\{\cdot,\cdot\}$ such that $(R,\{\cdot,\cdot\})$ is a Lie superalgebra and such that $\{\cdot,\cdot\}$ satisfies the Leibniz rule:
\begin{equation*}
\{x,yz\}=(-1)^{|x||y|}y\{x,z\}+\{x,y\}z
\end{equation*}
for all homogeneous $x,y,z\in R$. That is, $\{x,\cdot\}$ is a superderivation of degree $|x|$.
\end{definition}

\begin{example}
If $A$ is a supercommutative superalgebra, then $A_L$ is a Poisson superalgebra; note the bracket in this case is trivial. Moreover, any Poisson algebra $R$ can be considered a Poisson superalgebra via the trivial grading $R=R\oplus 0$.
\end{example}

To form the category of Poisson superalgebras, which we denote \textbf{Poiss}, we define morphisms as follows.

\begin{definition}\label{defn.Poisson_hom}
For Poisson superalgebras $R,P$, an even superalgebra homomorphism $\phi:R\rightarrow P$ is a \emph{super Poisson homomorphism} if
\begin{equation*}
\phi(\{x,y\})=\{\phi(x),\phi(y)\}
\end{equation*}
for all $x,y\in R$.
\end{definition}

\begin{remark}\label{rem.Poisson_antihom}
Just as we can define super Poisson homomorphisms between Poisson superalgebras $R,P$, we can define super Poisson \emph{anti-homomorphisms} as even superalgebra homomorphisms $\phi:R\rightarrow P$ satisfying
\begin{equation*}
    \phi(\{x,y\})=-\{\phi(x),\phi(y)\}=(-1)^{|x||y|}\{\phi(y),\phi(x)\}
\end{equation*}
for $x,y\in R$. This notion will be important when we discuss Poisson Hopf superalgebras in Section \ref{sec.Poisson_Hopf}.
\end{remark}

\begin{remark}
Though the requirement that super Poisson (anti)-homomorphisms be even maps may seem unnecessary at the moment, we need this condition to ensure the veracity of Lemma \ref{lemma.prop_P_composition} and Lemma \ref{lemma.U_functor}.
\end{remark}

We can also enrich the notion of superderivations in the Poisson setting; we will use this in Section \ref{sec.Poisson_Ore} when studying Poisson-Ore extensions.

\begin{definition}\label{defn.Poisson_superder}
Let $R$ be a Poisson superalgebra, let $\sigma:R\rightarrow R$ be a super Poisson homomorphism, and let $\alpha:R\rightarrow R$ be a homogeneous superderivation of $R$. Then $\alpha$ is a \emph{Poisson $\sigma$-superderivation} (of degree $|\alpha|$) if
\begin{equation*}
    \alpha(\{x,y\})=\{\alpha(x),\sigma(y)\}+(-1)^{|x|(|\sigma|+|\alpha|)+|\sigma||\alpha|}\{\sigma(x),\alpha(y)\}
\end{equation*}
for $x,y\in R$. A \emph{Poisson $\sigma$-superderivation} is the sum of an even Poisson $\sigma$-superderivation and an odd Poisson $\sigma$-superderivation; a \emph{Poisson superderivation} is a Poisson $\id_R$-superderivation.
\end{definition}

\begin{remark}
We will not need the full generality of the above definition in this paper. Rather, we only need the notion of a Poisson superderivation (i.e. the $\sigma=\id_R$ case).
\end{remark}

The remainder of this section will be devoted to introducing two interesting classes of Poisson superalgebras.

\subsection{Poisson Symplectic Superalgebras}\label{sec.Symplectic_superalgs}
(C.f. \cite[Example 2.2]{LOV_PBW}) Let $L$ be a Lie superalgebra with bracket $[\cdot, \cdot]$, and let $\sigma$ be a 2-cocycle in the trivial Lie superalgebra cohomology of $L$. That is, $\sigma:L\times L\rightarrow\kk$ is a bilinear form satisfying
\begin{enumerate}
    \item Super skew-symmetry: $\sigma(x,y)=-(-1)^{|x||y|}\sigma(y,x)$ for $x,y\in L$.
    \item Super Jacobi identity: $\sigma([x,y],z)+(-1)^{|z|(|x|+|y|)}\sigma([z,x],y)+(-1)^{|x|(|y|+|z|)}\sigma([y,z],x)=0$ for $x,y,z\in L$.
    \item Grading-preserving: $\sigma(x,y)=0$ for $|x|+|y|=1$.
\end{enumerate}
Define a bracket on the supersymmetric superalgebra $S(L)$ by
\begin{equation*}
    \{x,y\}_\sigma:=[x,y]+\sigma(x,y)
\end{equation*}
for $x,y\in L$, extending to $S(L)$ as a bisuperderivation; the case $\sigma=0$ has also been studied for the non-super case in, e.g., \cite[Example 11]{Oh_PoissonHopf} and \cite[Proposition 6.3]{LWZ_PoissonHopf}. One can verify this makes $S(L)$ into a Poisson superalgebra which we denote by $PS_\sigma(L)$, or simply $PS(L)$ if $\sigma=0$ is the trivial cocycle. We call these Poisson superalgebras \emph{Poisson symplectic superalgebras}.

In particular, let $L_n$ be an abelian Lie superalgebra with a basis $\{x_i,y_j\}_{1\leq i,j\leq n}$ consisting of odd elements. Then the bilinear map $\sigma:L_n\times L_n\rightarrow\kk$ defined by
\begin{equation*}
\sigma(x_i,y_j):=\sigma(y_i,x_j):=\delta_{ij}, \quad \sigma(x_i,x_j):=\sigma(y_i,y_j):=0
\end{equation*}
is a 2-cocycle. We call $PS_\sigma(L_n)$ the \emph{$n$-th Poisson symplectic superalgebra} and denote it $P_n$; these algebras are the super analogue of the $n$-th Poisson symplectic algebras $S_n$. It was shown in \cite{Umirbaev_Enveloping_Algebras} that the universal enveloping algebra $U(S_n)$ is isomorphic to the $2n$-th Weyl algebra $A_{2n}$. We will compute the universal enveloping algebra of $PS_\sigma(L)$ in the next section and show that in particular, $U(P_n)$ is a Weyl superalgebra.

\subsection{Quadratic Poisson Brackets on Exterior Algebras}
Let $P=\kk[x_1,\ldots,x_n]$ be a polynomial algebra with $\deg(x_i)=1$ for $1\leq i\leq n$. Following \cite{Sklyanin_PoissonPoly}, we say that $P$ is a \emph{quadratic Poisson polynomial algebra} if $P$ is given a Poisson bracket $\{\cdot,\cdot\}$ such that for each $x_i,x_j$, the bracket satisfies
\begin{equation*}
    \{x_i,x_j\}=\sum_{r,\ell}C_{r,l}^{i,j}x_rx_\ell
\end{equation*}
for some scalars $C_{r,\ell}^{i,j}$ satisfying $C_{r,\ell}^{i,j}=C_{\ell,r}^{i,j}=-C_{\ell,r}^{j,i}=-C_{r,\ell}^{j,i}$. We give the exterior algebra $P^!=\Lambda(\theta_1,\ldots,\theta_n)$ the $\ZZ_2$-grading defined by $|\theta_i|=1$ for $1\leq i\leq n$ and define the \emph{dual bracket} on $P^!$ by
\begin{equation*}
    \{\theta_r,\theta_\ell\}:=\sum_{i,j}C_{r,\ell}^{i,j}\theta_j\theta_i.
\end{equation*}
This bracket gives $P^!$ the structure of a Poisson superalgebra. Due to the form of the bracket, we will call $P^!$ (and $P$) a \emph{quadratic Poisson polynomial superalgebra}. Note that if $P$ is given the trivial bracket $\{x_i,x_j\}=0$, then the dual bracket on $P^!$ is also trivial. We remark that our convention is different than \cite{CFF_DualBracket}. In particular, the brackets are negatives of one another.

\begin{example}\label{eg.skew_symm} 
Let $(\lambda_{ij})$ be an $n\times n$ skew-symmetric matrix. Then $P=\kk[x_1,\ldots,x_n]$ can be made into a Poisson algebra via the bracket $\{x_i,x_i\}=\lambda_{ij}x_ix_j$ (see e.g. \cite[Example 2.2]{Oh_PoissonOre}). This is a quadratic Poisson polynomial algebra with $C_{i,j}^{i,j}=\lambda_{ij}$ and $C_{r,\ell}^{i,j}=0$ for $\{r,\ell\}\neq \{i,j\}$. In this case, the dual bracket on $P^!=\Lambda(\theta_1,\ldots,\theta_n)$ is given by $\{\theta_r,\theta_\ell\}=\lambda_{r\ell}\theta_\ell\theta_r$.
\end{example}

\section{Universal Enveloping Algebras}\label{sec.UEA}
In this section we will define and construct the universal enveloping algebra of a Poisson superalgebra, prove important results about universal enveloping algebras including the correspondence between modules over a Poisson superalgebra and modules over its universal enveloping algebra, and compute the universal enveloping algebra of Poisson symplectic superalgebras and quadratic Poisson polynomial superalgebras.

First, we need the notion of a Poisson supermodule over a Poisson superalgebra. 

\begin{definition}\label{defn.Poisson_module}
Let $R$ be a Poisson superalgebra and let $W=W_0\oplus W_1$ be a super vector space. We say $W$ is a \emph{(left) Poisson $R$-supermodule} if there is an even superalgebra homomorphism $\alpha:R\rightarrow\End(W)$ together with an even Lie superalgebra homomorphism $\beta:R\rightarrow\End(W)_L$ such that for all homogeneous $x,y\in R$,
\begin{enumerate}
    \item $\alpha(\{x,y\})=[\beta(x),\alpha(y)]_{gr}$,
    \item $\beta$ is an even $\alpha$-superderivation; i.e.,
    $\beta(xy)=\alpha(x)\beta(y)+(-1)^{|x||y|}\alpha(y)\beta(x)$.
\end{enumerate}
\end{definition}

\begin{example}
A Poisson superalgebra $R$ is naturally a Poisson $R$-supermodule via its multiplication and bracket. More generally, let $I$ be a \emph{Poisson superideal} of $R$; that is, $I$ is a $\ZZ_2$-graded ideal of $R$ such that $\{R,I\}\subseteq I$. Then $I$ is a Poisson $R$-supermodule again via the multiplication and bracket of $R$. Finally, we will see in Definition \ref{defn.enveloping_algebra} that the universal enveloping algebra of $R$ is also a Poisson $R$-supermodule.
\end{example}

In order to simplify the definition of Poisson supermodules, as well as the definition of the universal enveloping algebra, we introduce the following notation.

\begin{definition}\label{defn.prop_P}
Let $R$ be a Poisson superalgebra. We say a triple $(U,\alpha,\beta)$ \emph{satisfies property \textbf{P}} (with respect to $R$) if 
\begin{enumerate}
    \item $U$ is a superalgebra,
    \item $\alpha:R\rightarrow U$ is an even superalgebra homomorphism, and
    \item $\beta:R\rightarrow U_L$ is an even Lie superalgebra homomorphism,
\end{enumerate}
such that
\begin{enumerate}
    \item[(i)] $\alpha(\{x,y\})=[\beta(x),\alpha(y)]_{gr}$ for homogeneous $x,y\in R$,
    \item[(ii)] $\beta$ is an even $\alpha$-superderivation.
\end{enumerate}
In particular, a super vector space $W$ is a Poisson $R$-supermodule if and only if $(\End W,\alpha,\beta)$ satisfies property \textbf{P} for some $\alpha,\beta$.
\end{definition}

We are now ready to define and construct the universal enveloping algebra following \cite{huebschmann_Poisson/LR}, \cite{Oh_UEA}, \cite{Towers_PossonCohomology}.

\begin{definition}\label{defn.enveloping_algebra}
Let $R$ be a Poisson superalgebra. The \emph{universal enveloping algebra} (also called Poisson enveloping algebra or just enveloping algebra) of $R$ is a triple $(U(R),m,h)$—which we will often denote simply $U(R)$—that is universal with respect to property \textbf{P}. That is, $(U(R),m,h)$ satisfies property \textbf{P} and if $(B,\gamma,\delta)$ is another triple satisfying property \textbf{P}, then there is a unique (necessarily even) superalgebra homomorphism $\phi:U(R)\rightarrow B$ such that the following diagram commutes:
\begin{equation*}
\xymatrixcolsep{4pc}\xymatrixrowsep{4pc}\xymatrix{
   R \ar@<2pt>[r]^{m} \ar@<-2pt>[r]_{h} \ar@<2pt>[rd]_{\delta} \ar@<-2pt>[rd]^{\gamma} & U(R) \ar@{-->}[d]^{\phi} \\ & B
}
\end{equation*}
\end{definition}

\begin{theorem}\label{thm.enveloping algebra}
Every Poisson superalgebra $R$ has a unique universal enveloping algebra up to unique isomorphism.
\end{theorem}
\begin{proof}
To construct an enveloping algebra $U(R)$, let $M=\{m_x \mid x\in R\}$ and $H=\{h_x \mid x\in R\}$ be two super vector space copies of $R$ with the obvious even linear isomorphisms. Consider the tensor superalgebra $T=T(M\oplus H)$. Let $J$ be the two-sided ideal of $T$ generated by elements of the form
\begin{itemize}
    \item $m_xm_y-m_{xy}$
    \item $h_xm_y-(-1)^{|x||y|}m_yh_x-m_{\{x,y\}}$
    \item $m_1-1$
    \item $h_xh_y-(-1)^{|x||y|}h_yh_x-h_{\{x,y\}}$
    \item $m_xh_y+(-1)^{|x||y|}m_yh_x-h_{xy}$
\end{itemize}
for $x,y\in R$, where elements of $M\oplus H$ are identified with their canonical images in $T$. We claim $(T/J,m,h)$, where $m(x):=m_x+J$ and $h(x):=h_x+J$, is the universal enveloping algebra of $R$. That the above triple satisfies property \textbf{P} is immediate from the definition, so suppose $(B,\gamma,\delta)$ is another triple satisfying property \textbf{P}. Define a superalgebra homomorphism $\Phi:T\rightarrow B$ by $\Phi(m_x):=\gamma(x)$ and $\Phi(h_x):=\delta(x)$ for $x\in R$. Then
\begin{align*}
\Phi(m_1-1) &= \gamma(1)-1=0,\\
\Phi(m_xm_y-m_{xy}) &= \gamma(x)\gamma(y)-\gamma(xy)=0,\\
\Phi(h_xm_y-(-1)^{|x||y|}m_yh_x-m_{\{x,y\}}) &= \delta(x)\gamma(y)-(-1)^{|x||y|}\gamma(y)\delta(x)-\gamma(\{x,y\})=0,\\
\Phi(h_xh_y-(-1)^{|x||y|}h_yh_x-h_{\{x,y\}}) &= \delta(x)\delta(y)-(-1)^{|x||y|}\delta(y)\delta(x)-\delta(\{x,y\})=0,\\
\Phi(m_xh_y+(-1)^{|x||y|}m_yh_x-h_{xy}) &= \gamma(x)\delta(y)+(-1)^{|x||y|}\gamma(y)\delta(x)-\delta(xy)=0.
\end{align*}
Hence we have a well-defined superalgebra homomorphism $\phi:U(R)\rightarrow B$ defined by $\phi\circ\pi:=\Phi$. That $\phi m=\gamma$ and $\phi h=\delta$ is obvious. Finally, if $\psi:U(R)\rightarrow B$ is another superalgebra homomorphism such that $\psi m=\gamma$ and $\psi h=\delta$, then $\phi=\psi$ on the generators of $U(R)$ so $\phi=\psi$ on all of $U(R)$. That is, $\phi$ is the unique map completing the diagram.

The uniqueness of the enveloping algebra up to unique isomorphism follows from the standard argument for universal objects.
\end{proof}

In the future, we will abuse notation and simply write $m(x)=m_x,h(x)=h_x$ for simplicity.

\begin{corollary}\label{cor.Poisson_module_corr}
There is a 1-1 correspondence between Poisson $R$-supermodules and $U(R)$-supermodules.
\end{corollary}
\begin{proof}
Suppose $W$ is a Poisson $R$-supermodule. Then the universal property of $U(R)$ induces a superalgebra homomorphism $\phi:U(R)\rightarrow\End W$, giving $W$ a $U(R)$-supermodule structure.

Conversely, suppose $W$ is a $U(R)$-supermodule via a superalgebra homomorphism $\phi:U(R)\rightarrow\End W$. Then since $(U(R),m,h)$ satisfies property \textbf{P}, so does the triple $(\End W,\phi m,\phi h)$, thus making $W$ a Poisson $R$-supermodule.
\end{proof}

\begin{corollary}\label{cor.m_inj}
If $(U(R),m,h)$ is the universal enveloping algebra of a Poisson superalgebra $R$, then $m$ is injective.
\end{corollary}
\begin{proof}
For $x\in R$, let $\gamma_x$ denote left multiplication by $x$ and let $\delta_x$ denote the adjoint map $\delta_x(y)=\{x,y\}$. Define $\gamma,\delta:R\rightarrow\End R$ by
\begin{equation*}
\gamma(x):=\gamma_x, \qquad \delta(x):=\delta_x.
\end{equation*}
One easily verifies $(\End R, \gamma, \delta)$ satisfies property \textbf{P}, so there is an induced superalgebra homomorphism $\phi:U(R)\rightarrow\End R$. Therefore, if $x\in\text{ker}(m)$ then
\begin{equation*}
0=\phi m(x)=\gamma_x
\end{equation*}
and $\gamma_x(1)=x=0$. 
\end{proof}

The following two lemmas show that $U$ a functor from \textbf{Poiss} to the category \textbf{SAlg} of associative superalgebras.

\begin{lemma}
\label{lemma.prop_P_composition}
Let $R,P$ be Poisson superalgebras, and let $C$ be a superalgebra. If $\phi:R\rightarrow P$ is a super Poisson homomorphism and $(C,\gamma,\delta)$ satisfies property \textbf{P} with respect to $P$, then $(C,\gamma\phi,\delta\phi)$ satisfies property \textbf{P} with respect to $R$.
\end{lemma}
\begin{proof}
Straightforward.
\end{proof}

\begin{lemma}
\label{lemma.U_functor}
Let $(U_R,m_R,h_R), (U_P,m_P,h_P)$ be the universal enveloping algebras for Poisson superalgebras $R,P$, respectively. If $\phi:R\rightarrow P$ is a super Poisson homomorphism, then there is a unique superalgebra map $U(\phi):U(R)\rightarrow U(P)$ such that the following diagram commutes:
\begin{equation*}
\xymatrixcolsep{4pc}\xymatrixrowsep{4pc}\xymatrix{
    R \ar[r]^-{m_R,h_R} \ar[d]_{\phi} & U(R) \ar[d]^{U(\phi)} \\
    P \ar[r]_{m_P,h_P} & U(P)}
\end{equation*}
\end{lemma}
\begin{proof}
This follows immediately from the above lemma.
\end{proof}

\begin{definition}\label{defn.Poisson_tensor_prod}
Let $R,P$ be Poisson superalgebras with brackets $\{\cdot,\cdot\}_R,\{\cdot,\cdot\}_P$, respectively. Then the \emph{canonical bracket} on $R\stensor P$ is defined by
\begin{equation*}
    \{x\stensor z,y\stensor w\}_{R\stensor P}:=(-1)^{|z||y|}(\{x,y\}_R\stensor zw+xy\stensor\{z,w\}_P)
\end{equation*}
for $x,y\in R,z,w\in P$.
\end{definition}

One can check without much difficulty that the canonical bracket makes $R\stensor P$ into a Poisson superalgebra. Hence, hereafter whenever we refer to the \emph{Poisson superalgebra} $R\stensor P$ it will be endowed with the canonical bracket.

It is interesting to note how the functor $U$ behaves with respect to tensor products, and what the enveloping algebra of the opposite of a Poisson superalgebra is; we will define the latter after Proposition \ref{prop.tensor_UEA}. These results will also be of use to us in Section \ref{sec.Poisson_Hopf}. First, we need two technical lemmas.

\begin{lemma}\label{lem.prop_p_alt_condition}
Let $R$ be a Poisson superalgebra, let $B$ be a $\kk$-algebra, and let $m,h:R\rightarrow B$ be even linear maps satisfying
\begin{align*}
m(\{x,y\})=h(x)m(y)-(-1)^{|x||y|}m(y)h(x), \quad h(xy)=m(x)h(y)+(-1)^{|x||y|}m(y)h(x)
\end{align*}
for all $x,y\in R$. In particular, this holds if $(B,\gamma,\delta)$ satisfy property \textbf{P}. Then
\begin{equation*}
m(\{x,y\})=m(x)h(y)-(-1)^{|x||y|}h(y)m(x), \quad h(xy)=h(x)m(y)+(-1)^{|x||y|}h(y)m(x).
\end{equation*}
\end{lemma}
\begin{proof}
Since
\begin{align*}
&m(\{x,y\})+h(xy)=h(x)m(y)+m(x)h(y),\\
&m(\{y,x\})+h(yx)=h(y)m(x)+m(y)h(x),
\end{align*}
we have
\begin{align*}
2h(xy) &= h(x)m(y)+m(x)h(y)+(-1)^{|x||y|}(h(y)m(x)+m(y)h(x))\\
&= h(xy)+h(x)m(y)+(-1)^{|x||y|}h(y)m(x),\\
2m(\{x,y\}) &= h(x)m(y)+m(x)h(y)-(-1)^{|x||y|}(h(y)m(x)+m(y)h(x))\\
&= m(\{x,y\})+m(x)h(y)-(-1)^{|x||y|}h(y)m(x),
\end{align*}
from which the conclusions follow.
\end{proof}

\begin{lemma}
\label{lemma.tensor_of_maps}
Let $(U(R),m_R,h_R),(U(P),m_P,h_P)$ be the universal enveloping algebras of Poisson superalgebras $R,P$, respectively. Then
\begin{enumerate}
    \item $m_R\stensor m_P:R\stensor P\rightarrow U(R)\stensor U(P)$ is an even superalgebra homomorphism, and
    \item $m_R\stensor h_P+h_R\stensor m_P:R\stensor P\rightarrow (U(R)\stensor U(P))_L$ is an even Lie superalgebra homomorphism.
\end{enumerate}
\end{lemma}
\begin{proof}
(i) That $m_R\stensor m_P$ is an even superalgebra homomorphism follows from the fact that $m_R$ and $m_P$ are.\\
(ii) For $x,y\in R$ and $z,w\in P$,
\begin{align*}
&(m_R\stensor h_P+h_R\stensor m_P)(\{x\stensor z,y\stensor w\}_{R\stensor P})-\Big[(m_R\stensor h_P+h_R\stensor m_P)(x\stensor z)(m_R\stensor h_P+h_R\stensor m_P)(y\stensor w)\\
    & \qquad\qquad -(-1)^{|x\stensor z||y\stensor w|}(m_R\stensor h_P+h_R\stensor m_P)(y\stensor w)(m_R\stensor h_P+h_R\stensor m_P)(x\stensor z)\Big]\\
&= (-1)^{|y||z|}\big(m_R(\{x,y\}_R)\stensor h_P(zw)+m_R(xy)\stensor h_P(\{z,w\}_P)+h_R(\{x,y\}_R)\stensor m_P(zw)+h_R(xy)\stensor m_P(\{z,w\}_P)\big)\\
    & \qquad\qquad -\Big[(m_R(x)\stensor h_P(z)+h_R(x)\stensor m_P(z))(m_R(y)\stensor h_P(w)+h_R(y)\stensor m_P(w))-(-1)^{|x\stensor z||y\stensor w|}(m_R(y)\stensor h_P(w)\\
    & \qquad\qquad +h_R(y)\stensor m_P(w))(m_R(x)\stensor h_P(z)+h_R(x)\stensor m_P(z))\Big] \\
&= (-1)^{|y||z|}\Bigg[(h_R(x)m_R(y)-(-1)^{|x||y|}m_R(y)h_R(x))\stensor(m_P(z)h_P(     w)+(-1)^{|z||w|}m_P(w)h_P(z))+m_R(xy)\stensor(h_P(z)h_P(w)\\
    &\qquad\qquad -(-1)^{|z||w|}h_P(w)h_P(z))+(h_R(x)h_R(y)-(-1)^{|x||y|}h_R(y)h_R(x))\stensor m_P(zw)+(m_R(x)h_R(y)\\
    &\qquad\qquad+(-1)^{|x||y|}m_R(y)h_R(x))\stensor(h_P(z)m_P(w)-(-1)^{|z||w|}m_P(w)h_P(z))-\bigg[m_R(xy)\stensor h_P(z)h_P(w)\\
    &\qquad\qquad +m_R(x)h_R(y)\stensor h_P(z)m_P(w)+h_R(x)m_R(y)\stensor m_P(z)h_P(w)+h_R(x)h_R(y)\stensor m_P(zw)\\
    &\qquad\qquad-(-1)^{|x||y|+|z||w|}\Big[m_R(yx)\stensor h_P(w)h_P(z)+m_R(y)h_R(x)\stensor h_P(w)m_P(z)\\
    &\qquad\qquad+h_R(y)m_R(x)\stensor m_P(w)h_P(z)+h_R(y)h_R(x)\stensor m_P(wz)\Big]\bigg]\Bigg]\\
&= -(-1)^{|y||z|+|x||y|}m_R(y)h_R(x)\stensor(m_P(z)h_P(w)-(-1)^{|z||w|}h_P(w)m_P(z))+(-1)^{|y||z|+|z||w|}(h_R(x)m_R(y)\\
    & \qquad\qquad -(-1)^{|x||y|}m_R(y)h_R(x))\stensor m_P(w)h_P(z)-(-1)^{|y||z|+|z||w|}(m_R(x)h_R(y)-(-1)^{|x||y|}h_R(y)m_R(x))\stensor m_P(w)h_P(z)\\
    & \qquad\qquad +(-1)^{|y||z|+|x||y|}m_R(y)h_R(x)\stensor(h_P(z)m_P(w)-(-1)^{|z||w|}m_P(w)h_P(z))\\
&= -(-1)^{|y||z|+|x||y|}m_R(y)h_R(x)\stensor m_P(\{z,w\}_P)+(-1)^{|y||z|+|z||w|}m_R(\{x,y\}_R)\stensor m_P(w)h_P(z)\\
    & \qquad\qquad -(-1)^{|y||z|+|z||w|}m_R(\{x,y\}_R)\stensor m_P(w)h_P(z)+(-1)^{|y||z|+|x||y|}m_R(y)h_R(x)\stensor m_P(\{z,w\}_P)\\
&= 0,
\end{align*}
where the penultimate equality follows from Lemma \ref{lem.prop_p_alt_condition}. Clearly $m_R\stensor h_P+h_R\stensor m_P$ preserves the grading, so it is indeed an even Lie superalgebra homomorphism.
\end{proof}

\begin{proposition}
\label{prop.tensor_UEA}
Let $(U(R),m_R,h_R),(U(P),m_P,h_P)$ be the universal enveloping algebras of Poisson superalgebras $R,P$, respectively. Then $(U(R)\stensor U(P),m_R\stensor m_P,m_R\stensor h_P+h_R\stensor m_P)$ is the universal enveloping algebra of the Poisson superalgebra $R\stensor P$.
\end{proposition}
\begin{proof}
First, note that
\begin{align*}
    (m_R\stensor m_P)(\{x\stensor z,y\stensor w\}_{R\stensor P}) &= (m_R\stensor h_P+h_R\stensor m_P)(x\stensor z)(m_R\stensor m_P)(y\stensor w)\\
    & \quad-(-1)^{|x\stensor z||y\stensor w|}(m_R\stensor m_P)(y\stensor w)(m_R\stensor h_P+h_R\stensor m_P)(x\stensor z)
\end{align*}
and
\begin{align*}
    (m_R\stensor h_P+h_R\stensor m_P)((x\stensor z)(y\stensor w)) &= (m_R\stensor m_P)(x\stensor z)(m_R\stensor h_P+h_R\stensor m_P)(y\stensor w)\\
    & \quad +(-1)^{|x\stensor z||y\stensor w|}(m_R\stensor m_P)(y\stensor w)(m_R\stensor h_P+h_R\stensor m_P)(x\stensor z).
\end{align*}
hold for all $x,y\in R$ and $z,w\in P$. Hence by Lemma \ref{lemma.tensor_of_maps}, the triple $(U(R)\stensor U(P),m_R\stensor m_P,m_R\stensor h_P+h_R\stensor m_P)$ satisfies property \textbf{P}.

Let $i_1:R\rightarrow R\stensor P$ and $i_2:P\rightarrow R\stensor P$ be the super Poisson homomorphisms defined by
\begin{align*}
&i_1(x):=x\stensor 1\\
&i_2(z):=1\stensor z
\end{align*}
for $x\in R$ and $z\in P$. Let $C$ be a superalgebra with multiplication map $\mu_C$. If $\gamma,\delta:R\stensor P\rightarrow C$ are maps such that $(C,\gamma,\delta)$ satisfies property \textbf{P} with respect to $R\stensor P$, then by Lemma \ref{lemma.prop_P_composition} there exist (even) superalgebra maps $f:U(R)\rightarrow C$ and $g:U(P)\rightarrow C$ such that the following diagrams commute
\begin{equation*}
\xymatrixcolsep{4pc}\xymatrixrowsep{4pc}\xymatrix{
    R \ar[r]^-{m_R,h_R} \ar[d]_{i_1} & U(R) \ar[d]^{f} \\
    R\stensor P \ar[r]_{\gamma,\delta} & C}
    \qquad\qquad
    \xymatrixcolsep{4pc}\xymatrixrowsep{4pc}\xymatrix{
    P \ar[r]^-{m_P,h_P} \ar[d]_{i_2} & U(P) \ar[d]^{g} \\
    R\stensor P \ar[r]_{\gamma,\delta} & C}
\end{equation*}
Moreover, we have
\begin{align*}
\delta i_1(x)\gamma i_2(z)-(-1)^{|x||z|}\gamma i_2(z)\delta i_1(x) &= \delta(x\stensor 1)\gamma(1\stensor z)-(-1)^{|x||z|}\gamma(1\stensor z)\delta(x\stensor 1)\\
&= \gamma(\{x\stensor 1,1\stensor z\}_{R\stensor P})\\
&= \gamma(\{x,1\}_R)\stensor z+x\stensor\gamma(\{1,z\}_P) = 0.
\end{align*}
Therefore,
\begin{align*}
\mu_C(f\stensor g)(m_R\stensor m_P)(x\stensor z) &= fm_R(x)gm_P(z)=\gamma(i_1(x)i_2(z))=\gamma(x\stensor z) \\
\mu_C(f\stensor g)(m_R\stensor h_P+h_R\stensor m_P)(x\stensor z) &= fm_R(x)gh_P(z)+gh_R(x)gm_P(z)\\
&= \gamma i_1(x)\delta i_2(z)+\delta i_1(x)\gamma i_2(z)\\
&= \gamma i_1(x)\delta i_2(z)+(-1)^{|x||z|}\gamma i_2(z)\delta i_1(x)\\
&= \delta(i_1(x)i_2(z))\\
&= \delta(x\stensor z).
\end{align*}
Thus $\mu_C(f\stensor g):U(R)\stensor U(P)\rightarrow C$ is a superalgebra map such that
\begin{equation*}
\mu_C(f\stensor g)(m_R\stensor m_P)=\gamma, \quad \mu_C(f\stensor g)(m_R\stensor h_P+h_R\stensor m_P)=\delta.
\end{equation*}
Now, if $\phi:U(R)\stensor U(P)\rightarrow C$ is another superalgebra map such that
\begin{equation*}
\phi(m_R\stensor m_P)=\gamma, \quad \phi(m_R\stensor h_P+h_R\stensor m_P)=\delta,
\end{equation*}
then
\begin{align*}
&\mu_C(f\stensor g)(m_R(x)\stensor 1)=\gamma(x\stensor 1)=\phi(m_R(x)\stensor 1)\\
&\mu_C(f\stensor g)(1\stensor m_P(z))=\gamma(1\stensor z)=\phi(1\stensor m_P(z))\\
&\mu_C(f\stensor g)(h_R(x)\stensor 1)=\delta(x\stensor 1)=\phi(h_R(x)\stensor 1)\\
&\mu_C(f\stensor g)(1\stensor h_P(z))=\delta(1\stensor z)=\phi(1\stensor h_P(z))
\end{align*}
for all $x\in R$ and $z\in P$. Finally, since $U(R)$ is generated by $m_R(R)$ and $h_R(R)$, and likewise for $U(P)$, we have $\mu_C(f\stensor g)=\phi$. Therefore, $(U(R)\stensor U(P),m_R\stensor m_P,m_R\stensor h_P+h_R\stensor m_P)$ is universal with respect to property \textbf{P}.
\end{proof}

Recall the opposite superalgebra of a superalgebra $B$ is denoted by $B^{\op}=(B,\circ)$. Let $(R,\cdot,\{\cdot,\cdot\})$ be a Poisson superalgebra. Then the \emph{opposite Poisson superalgebra} is $(R^{\op},\circ,\{\cdot,\cdot\}^{\op})$  with Poisson bracket given by $\{x,y\}^{\op}:=-\{x,y\}$.

\begin{proposition}
\label{prop.Opposite_UEA}
Let $(U(R),m,h)$ be the universal enveloping algebra of a Poisson superalgebra $R$. Then $(U(R)^{\op},m,h)$ is the universal enveloping algebra of $R^{\op}$.
\end{proposition}
\begin{proof}
We first show $(U(R),m,-h)$ is the universal enveloping algebra of $R^{\op}$. Indeed, suppose $(B,\gamma,\delta)$ satisfies property \textbf{P} with respect to $R^{\op}$. Then
\begin{align*}
\gamma(\{x,y\}) 
    &= -\gamma(\{x,y\}^{\op})=-\delta(x)\gamma(y)+(-1)^{|x||y|}\gamma(y)\delta(x)\\
    &= (-\delta)(x)\gamma(y)-(-1)^{|x||y|}\gamma(y)(-\delta)(x),\\
-\delta(\{x,y\})
    &=\delta(\{x,y\}^{\op})=\delta(x)\delta(y)-(-1)^{|x||y|}\delta(y)\delta(x)\\
    &= (-\delta)(x)(-\delta)(y)-(-1)^{|x||y|}(-\delta)(y)(-\delta)(x),\\
-\delta(xy)
    &= -\gamma(x)\delta(y)-(-1)^{|x||y|}\gamma(y)\delta(x)=\gamma(x)(-\delta)(y)+(-1)^{|x||y|}\gamma(y)(-\delta)(x),
\end{align*}
for $x,y\in R$, so $(B,\gamma,-\delta)$ satisfies property \textbf{P} with respect to $R$. Similarly, if $(B,\gamma,\delta)$ satisfies property \textbf{P} with respect to $R$, then $(B,\gamma,-\delta)$ satisfies property \textbf{P} with respect to $R^{\op}$. That $(U(R),m,-h)$ is the universal enveloping algebra of $R^{\op}$ follows immediately.

Consider now the superalgebra map
\begin{align*}
\phi:U(R) &\rightarrow U(R)^{\op}\\
m_x &\mapsto m_x\\
h_x &\mapsto -h_x.
\end{align*}
One easily shows this map is indeed well-defined as, e.g.,
\begin{align*}
\phi(h_xm_y-(-1)^{|x||y|}m_yh_x-m_{\{x,y\}})&=-h_x\circ m_y+(-1)^{|x||y|}m_y\circ h_x-m_{\{x,y\}}\\
&=h_xm_y-(-1)^{|x||y|}m_yh_x-m_{\{x,y\}}\\
&=0,
\end{align*}
with the other relations of $U(R)$ similarly mapping to 0. In the opposite direction, one can define a superalgebra map
\begin{align*}
\psi: U(R)^{\op} &\rightarrow U(R)\\
m_x &\rightarrow m_x\\
h_x &\rightarrow -h_x
\end{align*}
which is likewise well-defined. Clearly the $\phi$ and $\psi $ are inverses, proving $U(R)\cong U(R)^{\op}$. Further, is is clear $\phi h=-h$ and $\phi m=m$. It follows that $(U(R)^{\op},m,h)$ is indeed the universal enveloping algebra of $R^{\op}$, completing the proof.
\end{proof}

\subsection{Enveloping Algebras of Poisson Symplectic Superalgebras}
Consider the Poisson symplectic superalgebras $PS_\sigma(L)$ introduced in Section \ref{sec.Symplectic_superalgs}. Following \cite[\S 2.2.3]{LOV_PBW}, we will describe the universal enveloping algebra $U(PS_\sigma(L))$ as a \emph{Sridharan superalgebra} \cite[Definition 2.1]{Sridharan}.

Denote by $L^0$ the abelian Lie superalgebra on the underlying super vector space of $L$ and consider the adjoint action of $L$ on $L^0$:
\begin{equation*}
    \ad(x)(u):=[x,u]_L\in L^0
\end{equation*}
for $x\in L,u\in L^0$. Clearly this defines a Lie superalgebra homomorphism $\ad:L\rightarrow\Der(L^0)$ so we can form the semidirect product $\mathfrak{g}:=L^0\rtimes_{\ad} L$; recall this is the Lie superalgebra whose underlying super vector space is $L^0\oplus L$ with bracket given by
\begin{equation*}
    [u+x,v+y]:=(\ad(x)(v)-(-1)^{|u||y|}\ad(y)(u))+[x,y]_L
\end{equation*}
for $x,y\in L,u,v\in L^0$. For notational convenience, denote by $x^0$ (resp. $x^1$) the canonical image of $x\in L^0$ in $\mathfrak{g}$ (resp. the canonical image of $x\in L$ in $\mathfrak{g}$). Notice we can extend the cocycle $\sigma$ to a cocycle $\sigma_\mathfrak{g}$ on $\mathfrak{g}$ as follows:
\begin{equation*}
    \sigma_\mathfrak{g}(x^0,y^0):=\sigma_\mathfrak{g}(x^1,y^1):=0, \quad \sigma_\mathfrak{g}(x^0,y^1):=\sigma_\mathfrak{g}(x^1,y^0):=\sigma(x,y).
\end{equation*}
Now, consider the Sridharan superalgebra (also called the modified Lie enveloping (super)algebra) of the pair $(\mathfrak{g},\sigma_\mathfrak{g})$:
\begin{equation*}
    U_{\sigma_\mathfrak{g}}(\mathfrak{g}):=\frac{T(\mathfrak{g})}{J_\sigma},
\end{equation*}
$J_\sigma$ is the ideal generated by the elements
\begin{equation*}
x\stensor y-(-1)^{|x||y|}y\stensor x-[x,y]_\mathfrak{g}-\sigma_\mathfrak{g}(x,y)
\end{equation*}
for $x,y\in\mathfrak{g}$. Notice that such elements are homogeneous since $\sigma$ is grading-preserving. Hence $U_{\sigma_\mathfrak{g}}(\mathfrak{g})$ inherits the $\ZZ_2$-grading from $T(\mathfrak{g})$. Finally, let $\iota:\mathfrak{g}\rightarrow U_{\sigma_\mathfrak{g}}(\mathfrak{g})$ denote the canonical inclusion, let $\alpha:PS_\sigma(L)\rightarrow U_{\sigma_\mathfrak{g}}(\mathfrak{g})$ be the even superalgebra homomorphism defined by $x\mapsto\iota(x^0)$ for $x\in L$, and let $\beta:PS_\sigma(L)\rightarrow U_{\sigma_\mathfrak{g}}(\mathfrak{g})$ be the even $\alpha$-superderivation defined by $x\mapsto\iota(x^1)$ for $x\in L$. This is well-defined as for $x,y\in L$ we have
\begin{equation*}
    \alpha(x)\stensor\alpha(y)-(-1)^{|x||y|}\alpha(y)\stensor\alpha(x)=[x^0,y^0]_\mathfrak{g}+\sigma_\mathfrak{g}(x^0,y^0)=0,
\end{equation*}
and likewise for $\beta$.

\begin{proposition}[c.f. {\cite[Proposition 2.11]{LOV_PBW}}]\label{prop.symplectic_enveloping_alg}
The triple $(U_{\sigma_\mathfrak{g}}(\mathfrak{g}),\alpha,\beta)$ is the universal enveloping algebra of the Poisson symplectic superalgebra $PS_\sigma(L)$.
\end{proposition}
\begin{proof}
First, notice that for $x,y\in L$
\begin{align*}
    \alpha(\{x,y\}_\sigma)&=[x,y]^0+\sigma(x,y)=[\beta(x),\alpha(y)]_{gr},\\
    \beta(\{x,y\}_\sigma)&=[x,y]^1+\beta(\sigma(x,y))=[\beta(x),\beta(y)]_{gr},
\end{align*}
since $\alpha$ is a superalgebra homomorphism and $\beta$ is an $\alpha$-superderivation. It is not too difficult to show by induction then, that $(U_{\sigma_\mathfrak{g}}(\mathfrak{g}),\alpha,\beta)$ satisfies property \textbf{P}. That this triple is universal with respect to property \textbf{P} is simple using Lemma \ref{lem.prop_p_alt_condition}.
\end{proof}

\begin{example}
Consider the $n$-th symplectic Poisson superalgebra $P_n$. Recall this is the Poisson symplectic superalgebra $PS_\sigma(L_n)$ where $L_n$ is an abelian Lie superalgebra with a basis $\{x_i,y_j\}_{1\leq i,j\leq n}$ consisting of odd elements, and where $\sigma$ is defined by
\begin{equation*}
    \sigma(x_i,y_j):=\sigma(y_i,x_j):=\delta_{ij}, \quad \sigma(x_i,x_j):=\sigma(y_i,y_j):=0.
\end{equation*}
By the above proposition, $U(P_n)$ is the superalgebra with odd generators
\begin{equation*}
    x_1^0,\ldots,x_n^0,y_1^0,\ldots,y_n^0,x_1^1,\ldots,x_n^1,y_1^1,\ldots,y_n^1,
\end{equation*}
and relations
\begin{equation*}
    [x_i^\ep,x_j^{\ep'}]_{\gr}=[y_i^\ep,y_j^{\ep'}]_{\gr}=[x_i^0,y_j^0]_{\gr}=[x_i^1,y_j^1]_{\gr}=0,\, [x_i^0,y_j^1]_{\gr}=[x_i^1,y_j^0]_{\gr}=\delta_{ij}
\end{equation*}
for $1\leq i,j\leq n$ and $\ep,\ep'\in\{0,1\}$. By relabeling the variables
\begin{align*}
    x_i^0&\mapsto X_i,\\
    y_i^0&\mapsto X_{n+i},\\
    x_i^1&\mapsto Y_{n+i},\\
    y_i^1&\mapsto Y_i,
\end{align*}
we see $U(P_n)\cong C_{0\mid 2n}$, the \emph{Weyl superalgebra} of degree $0\mid 2n$ (see \cite[Definition 8]{Weyl_superalg}, and see \cite[Definition 3.1]{Weyl_superalg_2} for a more general treatment of Clifford/Weyl superalgebras).
\end{example}

\subsection{Enveloping Algebras of Quadratic Poisson Polynomial Superalgebras}
In \cite{Towers_PossonCohomology}, the universal enveloping algebras of quadratic Poisson polynomial algebras and their quadratic duals were studied. We briefly recall the results of interest to us and use them to find a presentation for the enveloping algebras of the Poisson superalgebras discussed in Example \ref{eg.skew_symm}.

By \cite[Lemma 2.2]{Towers_PossonCohomology}, the Poisson enveloping algebra $U(P)$ of $P$ is the triple $(R,m,h)$ where $R$ is the quadratic algebra $T(m_{x_i},h_{x_i})/I$ generated by the even elements $m_{x_i},h_{x_i}$, where $I$ is the two-sided ideal generated by
\begin{align*}
&m_{x_i}m_{x_j}-m_{x_j}m_{x_i},\\
&h_{x_i}m_{x_j}-m_{x_j}h_{x_i}-\sum_{r,\ell}C_{r,\ell}^{i,j}m_{x_r}m_{x_\ell},\\
&h_{x_i}h_{x_j}-h_{x_j}h_{x_i}-\sum_{r,\ell}C_{r,\ell}^{i,j}(m_{x_r}h_{x_\ell}+m_{x_\ell}h_{x_r}),
\end{align*}
for $1\leq i,j\leq n$, $m:P\rightarrow R$ is the algebra homomorphism defined by $m(x_i)=m_{x_i}$, and 
$h:P\rightarrow R$ is a linear map defined by $h(x_i)=h_{x_i}$ such that $h(xy)=m(x)h(y)+m(y)h(x)$ for $x,y\in P$. Similarly, by \cite[Corollary 4.6]{Towers_PossonCohomology}, the Poisson enveloping algebra $U(P^!)$ of $P^!$ is the triple $(S,m^!,h^!)$ where $S$ is the quadratic algebra $T(m_{\theta_r},h_{\theta_\ell})/J$ generated by the odd elements $m_{\theta_r},h_{\theta_\ell}$, where $J$ is two-sided ideal generated by
\begin{align*}
    &m_{\theta_r}m_{\theta_\ell}+m_{\theta_\ell}m_{\theta_r},\\
    &h_{\theta_r}m_{\theta_\ell}+m_{\theta_\ell}h_{\theta_r}-\sum_{i,j}C_{r,\ell}^{i,j}m_{\theta_j}m_{\theta_i},\\
    &h_{\theta_r}h_{\theta_\ell}+h_{\theta_\ell}h_{\theta_r}-\sum_{i,j}C_{r,\ell}^{i,j}(m_{\theta_j}h_{\theta_i}-m_{\theta_i}h_{\theta_j}),
\end{align*}
for $1\leq r,\ell\leq n$, $m^!:P^!\rightarrow S$ is the superalgebra homomorphism defined by $m(\theta_r)=m_{\theta_r}$, and $h^!:P^!\rightarrow S$ is the linear map defined by $h(\theta_r)=h_{\theta_r}$ such that $h(wz)=m(w)h(z)+(-1)^{|w||z|}m(z)h(w)$ for $w,z\in P^!$.

Additionally, it was shown both $U(P)$ and $U(P)^!$ are Koszul and moreover that $U(P^!)\cong U(P)^!$ \cite[Lemma 4.5 and Remark 4.7]{Towers_PossonCohomology}.

\begin{example}\label{eg.skew_symm_enveloping_alg}
Let $(\lambda_{ij})$ be an $n\times n$ skew-symmetric matrix. Consider the quadratic Poisson polynomial algebra $P=\kk[x_1,\ldots,x_n]$ with Poisson bracket
$\{x_i,x_i\}=\lambda_{ij}x_ix_j$, and consider its quadratic dual $P^!=\Lambda(\theta_1,\ldots,\theta_n)$ with dual Poisson bracket $\{\theta_r,\theta_\ell\}=\lambda_{r\ell}\theta_\ell\theta_r$ (Example \ref{eg.skew_symm}). Then by the aforementioned results of \cite{Towers_PossonCohomology}, we have the presentations
\begin{equation*}
    U(P)\cong\frac{\kk<m_{x_1},\ldots,m_{x_n},h_{x_1},\ldots,h_{x_n}>}{([m_{x_i},m_{x_j}],\, [h_{x_i},m_{x_j}]-\lambda_{ij}m_{x_i}m_{x_j},\, [h_{x_i},h_{x_j}]-\lambda_{ij}(m_{x_i}h_{x_j}+m_{x_j}h_{x_i}))}
\end{equation*}
and
\begin{equation*}
    U(P^!)\cong\frac{\kk<m_{\theta_1},\ldots,m_{\theta_n},h_{\theta_1},\ldots,h_{\theta_n}>}{([m_{\theta_k},m_{\theta_\ell}]_{\gr},\, [h_{\theta_k},m_{\theta_\ell}]_{\gr}-\lambda_{k\ell}m_{\theta_\ell}m_{\theta_k},\, [h_{\theta_k},h_{\theta_\ell}]_{\gr}-\lambda_{k\ell}(m_{\theta_\ell}h_{\theta_k}-m_{\theta_k}h_{\theta_\ell}))}.
\end{equation*}
\end{example}

\section{The PBW Theorem for Lie-Rinehart Superalgebras}\label{sec.LR_PBW}
In \cite{LWZ_PoissonHopf}, the authors prove a PBW Theorem for Poisson algebras via the PBW Theorem for Lie-Rinehart algebras. However, the version of the PBW theorem they use, proved by Rinehart in \cite{Rinehart_PBW}, contains a projectivity hypothesis which does not always hold for Poisson algebras; we discuss the translation of Lie-Rinehart PBW theorems to the Poisson setting in Section \ref{sec.Poisson_PBW}. In this section, we will prove a PBW Theorem for Lie-Rinehart superalgebras that avoids this restriction and which will be used in Section \ref{sec.Poisson_PBW} to prove a PBW theorem which holds for \emph{all} Poisson superalgebras (over a field). In this section alone, we work over a commutative ring $S$ of characteristic 0 rather than a field $\kk$.

\begin{definition}
A \emph{Lie-Rinehart superalgebra} is a pair $(A,L)$, where $A$ is a supercommutative $S$-superalgebra and $L$ is an $S$-Lie superalgebra as well as an $A$-supermodule, together with an even $S$-Lie superalgebra and $A$-supermodule homomorphism $\rho:L\rightarrow \Der(A)$ to the Lie superalgebra $\Der(A)$ of superderivations of $A$, such that for $x,y\in L, a\in A$,
\begin{equation*}
[x,ay] = (-1)^{|a||x|}a[x,y]+\rho(x)(a)y.
\end{equation*}
We will typically denote $\rho(x)(a)$ by $x(a)$ for simplicity, and we call $\rho$ the \emph{anchor map}.
\end{definition}

We now give some examples of Lie-Rinehart superalgebras.

\begin{example}\label{eg.LR}
\begin{enumerate}
    \item \label{LR1} If $A$ is any supercommutative superalgebra and $L$ is any $A$-Lie superalgebra, then we can give the pair $(A,L)$ the structure of a Lie-Rinehart superalgebra via the zero anchor: $x(a):=0$ for all $x\in L$ and $a\in A$. Conversely, if the pair $(A,L)$ is a Lie-Rinehart superalgebra with zero anchor map, then $L$ is an $A$-Lie superalgebra.

    Alternatively, if $A$ is any supercommutative superalgebra and $L$ is any $S$-Lie superalgebra which is also an $A$-supermodule, denote by $L^0$ the abelian $S$-Lie superalgebra on the underlying supermodule of $L$. Then the pair $(A,L^0)$ is given the structure of a Lie-Rinehart superalgebra via the zero anchor map.
    
    \item \label{LR2} For any supercommutative superalgebra $A$, the pair $(A,\Der(A))$ can be made into a Lie-Rinehart superalgebra via the identity anchor map.
    
    \item \label{LR3} If $R$ is a Poisson superalgebra, then $(R,R)$ is a Lie-Rinehart superalgebra with the natural action $x\cdot y=xy$ and anchor map $x\rightarrow \{x,\cdot\}$. In Section \ref{sec.Poisson_PBW}, we will describe another way to create a Lie-Rinehart superalgebra from a Poisson superalgebra that will be more suited to our purposes.
\end{enumerate}
\end{example}

By analogy with \cite{Rinehart_PBW}, we describe the universal enveloping algebra of a Lie-Rinehart superalgebra. Given a Lie-Rinehart superalgebra $(A,L)$ we can form the semidirect product $A_L\rtimes_\rho L$; note $A_L$ is an abelian Lie superalgebra since $A$ is supercommutative. Consider now the Lie enveloping algebra $U:=U(A_L\rtimes L)$ with inclusion map $\iota:A_L\rtimes L\rightarrow U$. Let $U^+$ denote the subalgebra of $U$ generated by $\iota(A_L\rtimes L)$. Lastly, let $P$ denote the two-sided ideal of $U^+$ generated by the elements $\iota(az)-\iota(a)\stensor \iota(z)$ for all $a\in A, z\in A_L\rtimes L$. Then the \emph{universal enveloping algebra} of $(A,L)$ is the quotient
\begin{equation*}
V(A,L):=U^+/P.
\end{equation*}
Note that $P$ is generated by homogeneous elements so $V(A,L)$ is in fact an $S$-superalgebra, as well as an $A$-supermodule, as it inherits a $\ZZ_2$-grading from $U^+$.

As with Poisson superalgebras, the enveloping algebra $V(A,L)$ has a universal property for which we need the following definition.

\begin{definition}
For a Lie-Rinehart superalgebra $(A,L)$, we say a triple $(U,f,g)$ \emph{satisfies property \textbf{R}} if
\begin{enumerate}
    \item $U$ is a superalgebra,
    \item $f:A\rightarrow U$ is an even superalgebra homomorphism,
    \item $g:L\rightarrow U_L$ is an even Lie superalgebra homomorphism,
\end{enumerate}
such that
\begin{enumerate}
\item[(i)] $f(x(a)) = [g(x),f(a)]_{gr}$,
\item[(ii)] $g(ax) = f(a)g(x)$,
\end{enumerate}
for all $a\in A,x\in L$.
\end{definition}

Continuing the analogy with Poisson superalgebras, we can use property \textbf{R} to define Lie-Rinehart supermodules over a Lie-Rinehart superalgebra.

\begin{definition}
A \emph{(left) Lie-Rinehart $(A,L)$-supermodule} is an $S$-supermodule $W$ together with maps $f:A\rightarrow\End(W)$ and $g:L\rightarrow\End(W)$ such that $(\End(W),f,g)$ satisfies property \textbf{R}. In other words, $W$ is an $L$-module as well as an $A$-supermodule, with compatibility conditions between the actions.
\end{definition}

The following is a straightforward generalization of the universal property of the enveloping algebra of a Lie-Rinehart algebra as described in \cite{huebschmann_Poisson/LR}.

\begin{proposition}
\label{prop.LR_UEA_universal_prop}
The triple $(V(A,L),\iota_A,\iota_L)$ is universal with respect to property \textbf{R}, where $\iota_A,\iota_L$ are the natural inclusions of $A,L$ in $V(A,L)$, respectively.
\end{proposition}
\begin{proof}
That the above triple satisfies property \textbf{R} is easy to verify. Let $(B,f,g)$ be another such triple. Since the underlying supermodule of $A_L\rtimes L$ is $A\oplus L$, we have a unique even linear map $h:A_L\rtimes L\rightarrow B$ that restricts to $f,g$ on $A,L$, respectively. Also, for homogeneous $a+x,b+y\in A_L\rtimes L$,
\begin{align*}
h([a+x,b+y]) &= f(x(b)-(-1)^{|a||y|}y(a))+g([x,y])\\
&= g(x)f(b)-(-1)^{|b||x|}f(b)g(x)-(-1)^{|a||y|}g(y)f(a)+f(a)g(y)\\
&+g(x)g(y)-(-1)^{|x||y|}g(y)g(x)
\end{align*}
while
\begin{align*}
h(a+x)h(b+y)-&(-1)^{|a||b|}h(b+y)h(a+x)\\ &= (f(a)+g(x))(f(b)+g(y))-(-1)^{|a||b|}(f(b)+g(y))(f(a)+g(x))\\
&= g(x)f(b)-(-1)^{|b||x|}f(b)g(x)-(-1)^{|a||y|}g(y)f(a)+f(a)g(y)\\
&+g(x)g(y)-(-1)^{|x||y|}g(y)g(x),
\end{align*}
where we use the fact that $|a|=|x|=|a+x|$ and $|b|=|y|=|b+y|$ since these elements are homogeneous. Therefore, by the universal property of $U(A_L\rtimes L)$ there is a unique superalgebra map $\phi$ making the following diagram commute:
\begin{equation*}
\xymatrixcolsep{4pc}\xymatrixrowsep{4pc}\xymatrix{
   A_L\rtimes L \ar@<2pt>[r]^{\iota} \ar@<2pt>[rd]_{h} & U(A_L\rtimes L) \ar@{-->}[d]^{\phi} \\ & B
}
\end{equation*}
In addition, the ideal $P$ is in the kernel of $\phi$ since for $a\in A$ and $z=b+x\in A_L\rtimes L$, we have
\begin{align*}
\phi(\iota(a)\iota(z)-\iota(az)) &= \phi(\iota(a))\phi(\iota(z))-\phi(\iota(az))\\
&= h(a)h(z)-h(az)\\
&= f(a)(f(b)+g(x))-h(az)\\
&= f(ab)+g(ax)-h(az)\\
&= h(a(b+x))-h(az)\\
&= 0.
\end{align*}
Therefore, there is a well-defined superalgebra map $\psi:V(A,L)\rightarrow B$ satisfying $\psi\iota_A=f,\psi\iota_L=g$, and it is easy to see this map is the unique map satisfying these conditions.
\end{proof}

The following is the Lie-Rinehart analogue of Corollary \ref{cor.Poisson_module_corr}.

\begin{corollary}
There is a 1-1 correspondence between Lie-Rinehart $(A,L)$-supermodules and $V(A,L)$-supermodules.
\end{corollary}

In contrast, the analogue of Corollary \ref{cor.m_inj} is true if and only if $A_L\rtimes L$ injects into its universal enveloping algebra. This holds in particular if $A_L\rtimes L$ is a free $S$-supermodule.

For the remainder of this section, we assume that for a Lie-Rinehart superalgebra $(A,L)$, both $A$ and $L$ are \emph{free} $S$-supermodules. For $p\geq 0$, let $V_p$ denote the left $A$-subsupermodule of $V(A,L)$ generated by products of at most $p$ elements of $L$, and let $V_{-1}:=0$. Then $\{V_p\}$ defines a filtration of $V(A,L)$ which we call the \emph{PBW filtration}. Denote by $\gr(V(A,L))$ the associated graded $A$-supermodule and note $w\stensor z-(-1)^{|w||z|}z\stensor w\in V_{p-1}$ for $w,z\in V(A,L)$ such that $w\stensor z\in V_p$, which follows from the relation $x\stensor y-(-1)^{|x||y|}y\stensor x=[x,y]$ for $x,y\in A_L\rtimes L$. In particular, this holds for $w\in A$ and $z\in V_p$ so $\gr(V(A,L))$ is in fact an $A$-superalgebra. Further, one can easily show $\gr(V(A,L))$ is supercommutative and thus there is a canonical $A$-superalgebra homomorphism $\phi:S_A(L)\rightarrow\gr(V(A,L))$ by the former's universal property; this map is easily seen to be surjective. We will prove this map is also injective and thus an $A$-superalgebra isomorphism.

To this end, denote $W:=U^+(A_L\rtimes L)$ and let $J$ be the two-sided ideal $J:=(a\stensor z-az)$ so $V:=V(A,L)=W/J$. Filtering $W$ by defining $W_p$ to be the $A$-subsupermodule generated by products of at most $p$ elements of $L$, we notice that the quotient $V=W/J$ is naturally filtered by $(W/J)_p=(W_p+J/J)$ and that this filtration coincides with the PBW filtration defined above. Hence
\begin{equation*}
\gr(V) = \bigoplus_{p=0}^\infty (W/J)_p\Big/(W/J)_{p-1}\cong\bigoplus_{p=0}^{\infty}W_p/(W_{p-1}+W_p\cap J)
\end{equation*}
with multiplication defined by $(x+T_{p-1})(y+T_{q-1})=xy+T_{p+q-1}$, where $T_p=W_{p-1}+W_p\cap J$.

One can similarly define $W':=S_R^+(A_L\rtimes L)$ and the two-sided ideal $I:=(a\stensor z-az)$ to form the quotient $V'=W'/I$. Notice that $V'=V(A,L^0)$, where $(A,L^0)$ is the Lie-Rinehart superalgebra defined in the second part of Example \ref{eg.LR}\eqref{LR1}. Further, $W'$ has an analogous filtration to $W$ with associated graded algebra
\begin{equation*}
\gr(V') = \bigoplus_{p=0}^{\infty}W'_p/(W'_{p-1}+W'_p\cap I)
\end{equation*}
having multiplication $(x'+T'_{p-1})(y'+T'_{q-1})=x'y'+T'_{p+q-1}$, where $T'_p=W'_{p-1}+W'_p\cap I$.

\begin{lemma}
\label{lem.indep of L bracket}
The $A$-superalgebras $\gr(V)$ and $\gr(V')$ are isomorphic.
\end{lemma}
\begin{proof}
Since $A$ and $L$ are free $S$-supermodules, there is a homogeneous $S$-basis of $A$ which we denote $\{x_i\}_{i=1}^n$, and a homogeneous $S$-basis of $L$ which we denote $\{y_j\}_{j=1}^m$. Then $A_L\rtimes L$ has an induced homogeneous $S$-basis $\{z_\ell\}_{\ell=1}^N$, where $N=n+m$, $z_i=x_i$ for $1\leq i\leq n$, and $z_j=y_{j-n}$ for $n+1\leq j\leq N$. By the PBW Theorem for Lie superalgebras, the map
\begin{align*}
\psi:\hspace{3pt} & S_S^+(A_L\rtimes L)\rightarrow U^+(A_L\rtimes L)\\
& z_1^{r_1}\stensor\cdots\stensor z_N^{r_N}\mapsto z_1^{r_1}\stensor\cdots\stensor z_N^{r_N},
\end{align*}
where $r_i\in\NN$ if $z_i$ is even and $r_i\in\{0,1\}$ if $z_i$ is odd, is an $S$-linear isomorphism. We claim $\psi$ is in fact a \emph{filtered} $S$-linear isomorphism. Indeed, first note that 
\begin{equation}\label{eq.filtration_eqn}
    w\stensor z-(-1)^{|w||z|}z\stensor w\in W_{p-1}
\end{equation}
for all $w,z\in U^+(A_L\rtimes L)$ such that $w\stensor z\in W_p$, which can be proven in the same manner as for the filtration $\{V_p\}$. Since $W'=S_R^+(A_L\rtimes L)$ is supercommutative it follows that the multiplications of $W$ and $W'$ are the same modulo lower degree terms, hence $\psi(W_p')\subseteq W_p$ as desired.

We next prove that $\psi$ induces an $S$-linear isomorphism between $\gr(V)$ and $\gr(V')$. We first show that $\psi(T_p')\subseteq T_p$ using equation \eqref{eq.filtration_eqn}. Indeed, by the above paragraph, it suffices to show an element of the form
\begin{equation*}
w=z_{i_1}\stensor\cdots\stensor z_{i_{\ell_1}}\stensor(x_j\stensor y_{j'}-x_jy_{j'})\stensor z_{i_{\ell_1+1}}\stensor\cdots\stensor z_{i_{\ell_1+\ell_2}}\in W_p'\cap I
\end{equation*}
is mapped into $T_p$. However, this is easy to see from the following process:
\begin{itemize}
    \item Permute the $z_i$ so that they are in ascending order;
    \item Apply $\psi$ to the resulting element;
    \item Undo the permutations in the first step and use equation \eqref{eq.filtration_eqn} to obtain
    \begin{equation*}
        \psi(w)=z_{i_1}\stensor\cdots\stensor z_{i_{\ell_1}}\stensor(x_j\stensor y_{j'}-x_jy_{j'})\stensor z_{i_{\ell_1+1}}\stensor\cdots\stensor z_{i_{\ell_1+\ell_2}} + \text{ lower degree terms},
    \end{equation*}
    which is clearly an element of $T_p=W_{p-1}+W_p\cap J$.
\end{itemize}
An analogous argument shows that $\psi^{-1}(T_p)\subseteq T_p'$, thus proving the restriction $\psi_p:W_p'\rightarrow W_p$ induces an $S$-linear isomorphism between $W_p'/T_p'$ and $W_p/T_p$. Thus $\gr(V)$ and $\gr(V')$ are isomorphic as $S$-modules. Once again using equation \eqref{eq.filtration_eqn}, one can easily show that this isomorphism is in fact an $A$-superalgebra isomorphism, completing the proof.
\end{proof}

We next show that the enveloping algebra of the Lie-Rinehart superalgebra $(A,L^0)$ is $S_A(L)$.

\begin{lemma}
\label{lemma.trivial_LR_UEA}
Let $(A,L)$ be a Lie-Rinehart superalgebra such that $L$ is abelian and such that anchor map $\rho$ is the zero map. Then the triple $(S_A(L),i_A,i_L)$, where $i_A:A\rightarrow S_A(L)$ and $i_L:L\rightarrow S_A(L)$ are the inclusion maps for $A$ and $L$, respectively, is universal with respect to property \textbf{R}. Consequently, $V(A,L)\cong S_A(L)$.
\end{lemma}
\begin{proof}
It is simple to verify the above triple satisfies property \textbf{R}, so it remains to show the triple is universal. Indeed, let $(B,f,g)$ be another triple satisfying property \textbf{R}. We need to show there is a unique $S$-superalgebra homomorphism $\phi:S_A(L)\rightarrow B$ such that $\phi i_A=f$ and $\phi i_l=g$. Uniqueness of $\phi$ is clear as $\phi(a)=f(a)$ and $\phi(x)=g(x)$ for $a\in A,x\in L$. To show $\phi$ is well-defined, it suffices to show $\phi(xy-(-1)^{|x||y|}yx)=0$ for $x,y\in A\cup L$. We show this for $x\in L$ and $y\in A$ and the leave other cases to the reader. Indeed,
\begin{align*}
\phi(xy-(-1)^{|x||y|}yx) & = \phi(x)\phi(y)-(-1)^{|x||y|}\phi(y)\phi(x)\\
& = g(x)f(y)-(-1)^{|x||y|}f(y)g(x)\\
& = f(x(y))\\
& = 0
\end{align*}
since the anchor map is zero.
\end{proof}

\begin{theorem}[Lie-Rinehart PBW Theorem]
\label{thm.LR_PBW}
The canonical $A$-superalgebra surjection $\phi:S_A(L)\rightarrow \gr(V(A,L))$ is an isomorphism.
\end{theorem}
\begin{proof}
By the above lemma, the map $\phi':S_A(L)\rightarrow V'$ sending $a\mapsto a$ and $x\mapsto x$, for $a\in A$ and $x\in L$, is an $S$-superalgebra isomorphism. Further, it is clear that $\phi'$ is in fact a filtered $A$-superalgebra isomorphism. Therefore, as $A$-superalgebras
\begin{equation*}
S_A(L)\cong\gr(V')\cong \gr(V(A,L))
\end{equation*}
with the latter isomorphism given by Lemma \ref{lem.indep of L bracket}. Clearly this composition is exactly the map specified in the statement of the theorem, so we are done.
\end{proof}

We end this section by briefly showing Theorem \ref{thm.LR_PBW} is independent of \cite[Theorem 3.1]{Rinehart_PBW}. Indeed, note that the above theorem does not apply to a Lie-Rinehart pair $(S,L)$ where $L$ is any projective, non-free Lie algebra over $S$. On the other hand, it \emph{always} applies whenever the base ring $S$ is a field, which we will use in the next section to prove the PBW Theorem holds for all Poisson superalgebras (over a field).

\section{The PBW Theorem for Poisson Superalgebras}\label{sec.Poisson_PBW}
In \cite{huebschmann_Poisson/LR}, it was shown that one can naturally view a Poisson algebra as a Lie-Rinehart algebra in such a way so that the enveloping algebra in the Poisson sense is isomorphic the the enveloping algebra in the Lie-Rinehart sense. Later, the authors of \cite{LWZ_PoissonHopf} used this viewpoint to prove a PBW theorem for Poisson algebras (\cite[Corollary 5.8]{LWZ_PoissonHopf}) that follows immediately from Rinehart's PBW theorem for Lie-Rinehart algebras (\cite[Theorem 3.1]{Rinehart_PBW}). We will use the same technique in this section, albeit with Theorem \ref{thm.LR_PBW}, to prove the PBW Theorem holds for all Poisson superalgebras over a field.

We first need to recall the construction of the \emph{even K\"{a}hler superdifferentials} over a superalgebra $A$; one can define the \emph{odd K\"{a}hler superdifferentials} similarly, though we will not need it in this paper.

Let $S$ be the free $A$-supermodule generated by the set $\{\d a \mid a\in A, \text{$a$ homogeneous}\}$ with grading $|\d a|=|a|$. Let $\Omega_A^{\ev}$ be the quotient of $S$ by the relations $\d(k a+\ell b)=k\d a+\ell\d b$ and $\d(ab)=a\d b+(-1)^{|a||b|}b\d a$, where $k,\ell\in \kk, a,b\in A$. Then the supermodule of even K\"{a}hler superdifferentials (also simply called the even K\"ahler superdifferentials) is the pair $(\Omega_A^{ev},d_{ev})$, where $d_{ev}$ is the even superderivation $d_{ev}(a):=\d a$.

Analogously to how the K\"{a}hler differentials over an algebra $A$ can be understood as a universal derivation, the even K\"{a}hler superdifferentials can be understood as a universal even superderivation.

\begin{proposition}[{\cite[Lemma 3.1]{MZ_Krull}, \cite[Chapter 3, \S 1.8]{Manin_Supermath}}]
\label{prop.Kahler_Differential}
Given any $A$-supermodule $M$, composition with $d_{\ev}$ gives a graded isomorphism (of abelian groups)
\begin{equation*}
\mathrm{Hom}_A(\Omega_A^{\ev},M)\cong\mathrm{Der}_\kk(A,M).
\end{equation*}
\end{proposition}

The following construction, analogous to the one in \cite{huebschmann_Poisson/LR}, allows us to view a Poisson superalgebra as a Lie-Rinehart superalgebra (c.f. Example \ref{eg.LR}\eqref{LR3}).

\begin{example}[{\cite[Theorem 3.8]{huebschmann_Poisson/LR}}]
Let $R$ be a Poisson superalgebra over $\kk$. Then the pair $(R,\Omega_R^{\ev})$ has the structure of a Lie-Rinehart superalgebra when $\Omega_R^{\ev}$ is given the bracket
\begin{equation*}
[x\d y,z\d w] := (-1)^{|y||z|}xz\d\{y,w\}+x\{y,z\}\d w-(-1)^{|xy||zw|}z\{w,x\}\d y,
\end{equation*}
for $x,y,z,w\in R$, with anchor map $\rho(\d z):=\{z,-\}$.
\end{example}

Notice that from the natural inclusions $\iota_R,\iota_{\Omega_R^{\ev}}$ of $R,\Omega_R^{\ev}$ in $V(R,\Omega_R^{\ev})$, respectively, we can define two natural maps from $R$ to $V(R,\Omega_R^{\ev})$: $\alpha:=\iota_R$ and $\beta:=\iota_{\Omega_R^{\ev}}\circ d_{ev}$. The following proposition shows the resulting triple is the Poisson enveloping algebra of $R$.

\begin{proposition}
\label{prop.Poisson_UEA=LR_UEA}
The triple $(V(R,\Omega_R^{\ev}),\alpha,\beta)$ is universal with respect to property \textbf{P}. In particular, there is a canonical superalgebra isomorphism $\Lambda:V(R,\Omega_R^{\ev})\rightarrow U(R)$.
\end{proposition}
\begin{proof}
We first need to show the triple satisfies property \textbf{P}. That $\alpha$ and $\beta$ are even and that $\alpha$ is a superalgebra homomorphism is trivial. Further, for homogeneous $x,y\in R$, 
\begin{align*}
&\beta(\{x,y\})=\d\{x,y\}=[\d x,\d y]_{R\rtimes \Omega_R^{ev}}=\beta(x)\beta(y)-(-1)^{|x||y|}\beta(y)\beta(x),\\
&\alpha(\{x,y\})=\{x,y\}=\rho(\d x)(y)=[\d x,y]_{V(R,\Omega_R^{ev})}=\beta(x)\alpha(y)-(-1)^{|x||y|}\alpha(y)\beta(x),\\
&\beta(xy)=\d(xy)=x\d y+(-1)^{|x||y|}y\d x=\alpha(x)\beta(y)+(-1)^{|x||y|}\alpha(y)\beta(x),
\end{align*}
so that the triple satisfies property \textbf{P}.

We now show $(V(R,\Omega_R^{ev}),\alpha,\beta)$ is universal with respect to property \textbf{P}. Indeed, suppose $(B,\gamma,\delta)$ is another triple satisfying property \textbf{P}. Note that $B$ is an $R$-supermodule via the action $x\cdot b=\gamma(x)b$ for $x\in R, b\in B$. Now, since
\begin{equation*}
    \delta(xy)=\gamma(x)\delta(y)+(-1)^{|x||y|}\gamma(y)\delta(x)=x\cdot\delta(y)+(-1)^{|x||y|}y\cdot\delta(x)
\end{equation*}
we have that $\delta:R\rightarrow B$ is an even $\kk$-superderivation. Thus by Proposition \ref{prop.Kahler_Differential}, there is a unique even $R$-supermodule homomorphism $\theta:\Omega_R^{ev}\rightarrow B$ such that $\delta=\theta d_{ev}$. Next, notice that the triple $(B,\gamma,\theta)$ satisfies property \textbf{R} (with respect to $(R,\Omega_R^{\ev}))$ since $\theta(x\d z)=\gamma(x)\delta(z)$ by definition, while
\begin{align*}
\gamma((x\d y)(z)) &= \gamma(x\{y,z\})\\
&= \gamma(x)\delta(y)\gamma(z)-(-1)^{|z||y|}\gamma(x)\gamma(z)\delta(y)\\
&= \gamma(x)\theta(\d y)\gamma(z)-(-1)^{|z||xy|}\gamma(z)\gamma(x)\theta(\d y)\\
&= \theta(x\d y)\gamma(z)-(-1)^{|z||xy|}\gamma(z)\theta(x\d y)
\end{align*}
with the condition that $\theta$ is a Lie superalgebra homomorphism being similarly easy to verify. Therefore, by the universal property of $V(R,\Omega_R^{ev})$, there is a unique superalgebra map $\lambda:V(R,\Omega_R^{ev})\rightarrow B$ such that $\lambda\alpha=\gamma$ and $\lambda\beta=\lambda\iota_{\Omega_R^{ev}}d_{ev}=\theta d_{ev}=\delta$. I.e., $(V(R,\Omega_R^{ev}),\alpha,\beta)$ is universal with respect to property \textbf{P}.

Finally, that there is a canonical superalgebra isomorphism between $V(R,\Omega_R^{ev})$ and $U(R)$ follows immediately from Theorem \ref{thm.enveloping algebra}.
\end{proof}

\begin{corollary}[Poisson PBW Theorem]
\label{cor.Poisson_PBW}
Let $R$ be a Poisson superalgebra, let $U(R)$ be its enveloping algebra, let $\Lambda$ be the canonical isomorphism obtained in the above proposition, and consider the filtration $\{F_n\}$ on $U(R)$ induced by the isomorphism $\Lambda$ and the PBW filtration on $V(R,\Omega_R^{\ev})$; we will also call this filtration the \emph{PBW filtration} (on $U(R)$). In particular, $m_x\in F_0$ and $h_x\in F_1$ for all $x\in R$. Then $\gr(U(R))$ is a supercommutative $R$-superalgebra, and the canonical superalgebra homomorphism from $S_R(\Omega_R^{\ev})$ to $\gr(U(R))$ is an $R$-superalgebra isomorphism.
\end{corollary}
\begin{proof}
That $\gr(U(R))$ is supercommutative follows from the fact $\gr(V(R,\Omega_R^{\ev}))$ is, while $S_R(\Omega_R^{\ev})\cong\gr(U(R))$ via the canonical homomorphism follows immediately from Theorem \ref{thm.LR_PBW}.
\end{proof}

\begin{corollary}\label{cor.U(R)_fg_Noeth_if_R_is}
Let $R$ be a Poisson superalgebra with enveloping algebra $U(R)$. If $R$ is finitely generated as a superalgebra, then $U(R)$ is left/right Noetherian.
\end{corollary}
\begin{proof}
Suppose $R$ is generated by the even variables $x_1,\ldots,x_n$ and the odd variables $y_1,\ldots,y_m$. Then $U(R)$ is generated by $m(R)$ and $h(R)$ and since
\begin{align*}
m(xy) &= m(x)m(y),\\
h(xy) &= m(x)h(y)+(-1)^{|x||y|}m(y)h(x),
\end{align*}
for $x,y\in R$, it follows $U(R)$ is generated by $m(x_i),m(y_j),h(x_i),h(y_j)$, for $1\leq i\leq n$, $1\leq j\leq m$, as a superalgebra.

Now, filter $U(R)$ via the PBW filtration and consider the associated graded superalgebra $\gr(U(R))$. Since $\gr(U(R))$ is supercommutative, it is the homomorphic image of a polynomial superalgebra $R[X_1,\cdots,X_n \mid Y_1,\cdots,Y_m]$ by the above paragraph. Since $R$ is finitely generated over $\kk$, it follows that $R[X_1,\cdots,X_n \mid Y_1,\cdots,Y_m]$, and hence $\gr(U(R))$, is left/right Noetherian by \cite[Lemma 1.4]{MZ_Krull}. By \cite[Theorem 1.6.9]{MR_NoetherianRings}, $U(R)$ is left/right Noetherian.
\end{proof}

In the case where $\Omega_R^{ev}$ is a free $R$-supermodule, we can use the above corollary to determine a basis of $U(R)$ as the following example demonstrates.

\begin{example}
\label{eg.symm_alg_UEA_basis}
Consider the polynomial superalgebra $R=\kk[x_1,\ldots,x_n \mid y_1,\ldots,y_m]$ on even generators $x_i$ and odd generators $y_j$, and endow $R$ with a Poisson bracket. Then $\Omega_R^{\ev}$ is the free $R$-supermodule generated by $\d x_1,\ldots,\d x_n,\d y_1,\ldots,\d y_m$ (see e.g. \cite[Remark 3.3]{MZ_Krull}). Therefore, by Corollary \ref{cor.Poisson_PBW}, $U(R)$ has an $R$-basis consisting of elements of the form
\begin{equation*}
h_{x_1}^{s_1}\cdots h_{x_n}^{s_n}h_{y_1}^{s_{n+1}}\cdots h_{y_m}^{s_{n+m}},
\end{equation*}
where $s_i$ is a nonnegative integer for $1\leq i\leq n$ and $s_j\in\{0,1\}$ for $n+1\leq j\leq n+m$. Thus $U(R)$ has a $\kk$-basis consisting of elements of the form
\begin{equation*}
m_{x_1}^{r_1}\cdots m_{x_n}^{r_n}m_{y_1}^{r_{n+1}}\cdots m_{y_m}^{r_{n+m}}
h_{x_1}^{s_1}\cdots h_{x_n}^{s_n}h_{y_1}^{s_{n+1}}\cdots h_{y_m}^{s_{n+m}},
\end{equation*}
where $r_i,s_i$ are nonnegative integers for $1\leq i\leq n$ and $r_j,s_j\in\{0,1\}$ for $n+1\leq j\leq n+m$. 

More generally, let $V$ be a super vector space with totally ordered homogeneous basis $\{x_i\}$, and endow the supersymmetric superalgebra $S(V)$ with a Poisson bracket. Then a similar argument to the above paragraph shows that $U(S(V))$ has a $\kk$-basis consisting of elements of the form
\begin{equation*}
m_{x_{i_1}}^{r_1}\cdots m_{x_{i_n}}^{r_n}h_{x_{j_1}}^{s_1}\cdots h_{x_{j_m}}^{s_m}
\end{equation*}
for all nonnegative $n,m$, where the $r_i,s_i$ are nonnegative (resp. 0 or 1) for $x_i$ even (resp. $x_i$ odd), where $x_{i_1}<\ldots<x_{i_n}$, and where $x_{j_1}<\ldots<x_{j_m}$.
\end{example}

\section{Enveloping algebras of Poisson Hopf superalgebras}\label{sec.Poisson_Hopf}
Poisson Hopf algebras were first introduced by Drinfeld in 1985 (\cite[Definition 1.1]{Drinfeld_PH}) and have since been studied in relation to homological algebra and deformation quantization, see e.g. \cite{BZ_PH,LW_PH,LWW_PH}. The universal enveloping algebras of Poisson Hopf algebras have been studied as well, for example in \cite{LWZ_PoissonHopf,Oh_PoissonHopf}. In particular, Oh showed in \cite{Oh_PoissonHopf} that the enveloping algebra of a Poisson Hopf algebra is itself a Hopf algebra. In this section we will prove an analogous result holds for Poisson superalgebras. We also show the Poisson symplectic superalgebra $PS(L)$ (i.e. with trivial cocycle) is a Hopf superalgebra, and describe the Hopf structure of its enveloping algebra.

Throughout this section, we will use sumless Sweedler notation for the comultiplication of a (super)coalgebra. We start with some definitions.

\begin{definition}[\cite{AM_Hopf_superalgebras}]
A \emph{supercoalgebra} is a $\ZZ_2$-graded coalgebra $C=C_0\oplus C_1$ such that $\Delta$ and $\ep$ preserve the grading: $\Delta(C_i)\subseteq \sum_jC_j\stensor C_{i-j}$, $\ep(C_0)\subseteq \kk$, and $\ep(C_1)=0$; the condition on $\ep$ follows from the fact fields are entirely even.

A \emph{superbialgebra} $B$ is a superalgebra and supercoalgebra, with respect to the same grading, such that $\Delta$ and $\ep$ are (even) superalgebra homomorphisms.

A \emph{Hopf superalgebra} is a superbialgebra $H$ admitting an antipode, which is an even linear map $S$ satisfying
\begin{equation*}
    \mu\circ(S\stensor\id_H)\circ\Delta=\mu\circ(\id_H\stensor S)\circ\Delta=\eta\circ\ep
\end{equation*}
where $\eta:\kk\rightarrow H$ is the unit map.
\end{definition}

\begin{definition}
\label{def.PHSuperalg}
A Poisson superalgebra $R$ is said to be a \emph{Poisson Hopf superalgebra} if $R$ is also a Hopf superalgebra $(R,\eta,\del,\ep,\Delta,S)$ such that $\Delta$ is a super Poisson homomorphism. That is,
\begin{equation*}
\Delta(\{x,y\})=\{\Delta(x),\Delta(y)\}_{R\stensor R}
\end{equation*}
for all $x,y\in R$.
\end{definition}

Note that the above definition only provides a compatibility condition between the comultiplication and the Poisson structure. The following lemma shows this is all that is necessary.

\begin{lemma}
\label{lemma.Poisson(anti)hom}
If $(R,\eta,\del,\ep,\Delta,S)$ is a Poisson Hopf superalgebra, then the counit $\ep$ is a super Poisson homomorphism and the antipode $S$ is a super Poisson anti-automorphism.
\end{lemma}
\begin{proof}
We first show $\ep$ is a super Poisson homomorphism. Indeed, since the bracket on $\kk$ is trivial this is true if and only if $\ep(\{x,y\})=0$ for all $x,y\in R$. Since $\ep$ is the counit, we have $\ep(x_1)\ep(x_2)=\ep(x)$ and thus
\begin{align*}
\ep(\{x,y\}) &= (-1)^{|y_1||x_2|}(\ep(\{x_1,y_1\})\ep(x_2y_2)+\ep(x_1y_1)\ep(\{x_2,y_2\}))\\
&= (-1)^{|y_1||x_2|}(\ep(\{x_1\ep(x_2),y_1\ep(y_2)\})+\ep(\{\ep(x_1)x_2,\ep(y_1)y_2\}))\\
&= (-1)^{|y_1||x_2|}(\ep(\{x,y\})+\ep(\{x,y\})),
\end{align*}
from which the result follows.

Consider now the antipode $S$. That $S$ is a superalgebra homomorphism follows from the fact that Poisson superalgebras are supercommutative and the discussion after Definition 1.5 of \cite{AM_Hopf_superalgebras}, which shows $S$ is a superalgebra anti-homomorphism; this discussion also shows $S$ is bijective. Now, from the equalities
\begin{align*}
&0=\{\ep(x),y\}=\{S(x_1)x_2,y\}=S(x_1)\{x_2,y\}+(-1)^{|x_2||y|}\{S(x_1),y\}x_2,\\
&0=\{x,\ep(y)\}=\{x,S(y_1)y_2\}=\{x,S(y_1)\}y_2+(-1)^{|x||y_1|}S(y_1)\{x,y_2\},
\end{align*}
we obtain
\begin{equation*}
S(x_1)\{x_2,y\}=-(-1)^{|x_2||y|}\{S(x_1),y\}x_2
\end{equation*}
and
\begin{equation*}
\{x,S(y_1)\}y_2=-(-1)^{|x||y_1|}S(y_1)\{x,y_2\}.
\end{equation*}
Also, by the first paragraph
\begin{equation*}
0=\ep(\{x,y\})=S(\{x,y\}_1)\{x,y\}_2=(-1)^{|y_1||x_2|}(S(\{x_1,y_1\})x_2y_2+S(x_1y_1)\{x_2,y_2\}),
\end{equation*}
from which we get
\begin{align*}
S(\{x_1,y_1\})x_2y_2 &= -S(x_1y_1)\{x_2,y_2\}\\
&= (-1)^{|x_1||y_1|}(-1)^{|x_2||y_2|}S(y_1)\{S(x_1),y_2\}x_2\\
&= -\{S(x_1),S(y_1)\}x_2y_2.
\end{align*}
using the relations just derived. Finally, we have
\begin{align*}
S(\{x,y\}) &= S(\{x_1\ep(x_2),y_1\ep(y_2)\})\\
&= S(\{x_1,y_1\})\ep(x_2)\ep(y_2)\\
&= S(\{x_1,y_1\})x_2S(x_3)y_2S(y_3)\\
&= (-1)^{|x_3||y_2|}S(\{x_1,y_1\})x_2y_2S(x_3)S(y_3)\\
&= -(-1)^{|x_3||y_2|}\{S(x_1),S(y_1)\}x_2y_2S(x_3)S(y_3)\\
&= -\{S(x_1),S(y_1)\}x_2S(x_3)y_2S(y_3)\\
&= -\{S(x_1\ep(x_2)),S(y_1\ep(y_2))\}\\
&= -\{S(x),S(y)\},
\end{align*}
completing the proof.
\end{proof}

\begin{example}
Let $G$ be an affine algebraic supergroup (see \cite[\S 11]{CCF_supergeometry}) and let $\kk(G)$ be the (finitely generated) Hopf superalgebra which represents $G$. If $\kk(G)$ is a Poisson Hopf superalgebra, we call $G$ a \emph{Poisson algebraic supergroup} (c.f. \cite[Definition 3.1.6]{KS_PH}). Note that in the special case where $G$ is a supervariety, \cite[Corollary 4.3]{Fioresi_supergroup} shows that $G$ is a Poisson-Lie supergroup.
\end{example}

Recall from Lemma \ref{lemma.U_functor} that $U:\textbf{Poiss}\rightarrow\textbf{SAlg}$ is a functor. In particular, given a Poisson Hopf superalgebra $(R,\eta,\del,\ep,\Delta,S)$, we can consider the superalgebra homomorphisms $U(\ep),U(\Delta),U(S)$.

\begin{theorem}
\label{thm.Hopf_UEA}
If $(R,\eta,\del,\ep,\Delta,S)$ is a Poisson Hopf superalgebra, then
\begin{equation*}
(U(R),\eta_{U(R)},\del_{U(R)},U(\ep),U(\Delta),U(S))
\end{equation*}
is a Hopf superalgebra such that
\begin{align*}
U(\Delta)m&=(m\stensor m)\Delta, & U(\Delta)h&=(m\stensor h+h\stensor m)\Delta,\\
U(\ep)m&=\ep, & U(\ep)h&=0,\\
U(S)m&=m S, & U(S)h&=h S.
\end{align*}
\end{theorem}
\begin{proof}
Since $\Delta:R\rightarrow R\stensor R$ is a super Poisson homomorphism and $(U(R)\stensor U(R),m\stensor m,m\stensor h+h\stensor m)$ is the Poisson enveloping algebra of $R\stensor R$ by Proposition \ref{prop.tensor_UEA}, Lemma \ref{lemma.U_functor} shows the superalgebra homomorphism $U(\Delta):U(R)\rightarrow U(R)\stensor U(R)$ satisfies
\begin{equation*}
U(\Delta)m=(m\stensor m)\Delta, \quad U(\Delta)h=(m\stensor h+h\stensor m)\Delta.
\end{equation*}
Similarly, the superalgebra homomorphism $U(\ep):U(R)\rightarrow U(\kk)=\kk$ satisfies $U(\ep)m=\ep$ and $U(\ep)h=0$, as $\ep$ is a super Poisson homomorphism by Lemma \ref{lemma.Poisson(anti)hom}. Moreover, since the antipode $S$ is a super Poisson homomorphism from $R$ to $R^{\op}$ by Lemma \ref{lemma.Poisson(anti)hom}, the superalgebra homomorphism $U(S)$ is a map from $U(R)$ to $U(R)^{\op}$ by Proposition \ref{prop.Opposite_UEA}. Additionally, $U(S)m=m S$ and $U(S)h=h S$ by Lemma \ref{lemma.U_functor}. Finally, it is routine to verify that $(U(R),\eta_{U(R)},\del_{U(R)},U(\ep),U(\Delta),U(S))$ is indeed a Hopf superalgebra.
\end{proof}

\begin{example}[{\cite[Example 11]{Oh_PoissonHopf}} and {\cite[Proposition 6.3]{LWZ_PoissonHopf}}]
Fix a Lie superalgebra $L$ over $\kk$. Recall the Poisson symplectic superalgebra $PS_\sigma(L)$ introduced in Section \ref{sec.Symplectic_superalgs} and consider the natural Hopf superalgebra structure on $S(L)$:
\begin{equation*}
\Delta(x):=x\stensor 1+1\stensor x, \quad \ep(x) := 0, \quad S(x):=-x,
\end{equation*}
for $x\in L$. If (and only if) $\sigma=0$ is the trivial cocycle, $\Delta(\{x,y\}_\sigma)=\{\Delta(x),\Delta(y)\}_{PS_\sigma(L)^{\stensor 2}}$ so that $(PS_\sigma(L),\eta,\del,\ep,\Delta,S)$ is a Poisson Hopf superalgebra. In particular, the Poisson enveloping algebra of $PS(L)$ is a Hopf superalgebra by Theorem \ref{thm.Hopf_UEA}. In fact, one can easily show this Hopf superalgebra structure is the same as the natural Hopf superalgebra structure on $U(PS(L))=U(L^0\rtimes L)$.
\end{example}

\section{Poisson-Ore Extensions of Poisson superalgebras and their Enveloping Algebras}\label{sec.Poisson_Ore}
Ore extensions are an important construction in ring theory arising naturally in several contexts such as universal enveloping algebras of solvable Lie algebras and coordinate rings of quantum groups. In \cite{Oh_PoissonOre}, Oh defined a similar construction for Poisson algebras, which we call Poisson-Ore extensions, though they were also studied in a limited way earlier by Polishchuk in \cite{polishchuk1997algebraic}. Later in \cite{LWZ_PoissonOre}, the authors showed that for a Poisson-Ore extension $R'$ of a Poisson algebra $R$, the enveloping algebra $U(R')$ is an iterated Ore extension of the enveloping algebra $U(R)$. In this section, we will generalize these results by defining Poisson-Ore extensions of a Poisson superalgebra $R$ by an even indeterminate $x$, and show that the enveloping algebra of a Poisson-Ore extension is an iterated Ore extension.

\begin{theorem}
\label{thm.Poisson_Ore_extension}
Let $\alpha,\delta$ be linear maps on a Poisson superalgebra $R$ with bracket $\{\cdot,\cdot\}_R$. Then the polynomial superalgebra $R[x]$, where $x$ is assumed to be even, is a Poisson superalgebra with bracket satisfying
\begin{equation}\label{eq.Poisson_Ore_bracket}
\{r,s\}=\{r,s\}_R, \quad \{x,r\}=\alpha(r)x+\delta(r)
\end{equation}
for $r,s\in R$, if and only if $\alpha$ is an even Poisson superderivation and $\delta$ is an even superderivation such that
\begin{equation}\label{eq.delta_condition}
\delta(\{r,s\}_R)=\alpha(r)\delta(s)-\delta(r)\alpha(s)+\{r,\delta(s)\}_R+\{\delta(r),s\}_R
\end{equation}
for $r,s\in R$. In this case, we denote the Poisson superalgebra $R[x]$ by $R[x;\alpha,\delta]_p$.
\end{theorem}
\begin{proof}
If $R[x]$ is a Poisson superalgebra, then $|\{x,r\}|=|r|=|\alpha(r)|=|\delta(r)|$ so $\alpha,\delta$ are even linear maps. Also,
\begin{align*}
\{x,rs\} &= \alpha(rs)x+\delta(rs)=
r\{x,s\}+\{x,r\}s=(r\alpha(s)+\alpha(r)s)x+r\delta(s)+\delta(r)s
\end{align*}
for $r,s\in R$, so $\alpha$ and $\delta$ are both even superderivations on $R$. Moreover, by the super Jacobi identity
\begin{align*}
0 &= \{x,\{r,s\}\}+\{r,\{s,x\}\}+(-1)^{|r||s|}\{s,\{x,r\}\}\\
&= (\alpha(\{r,s\}_R)-\{r,\alpha(s)\}-\{\alpha(r),s\})x\\
& \quad +\delta(\{r,s\}_R)-\{r,\delta(s)\}_R-\{\delta(r),s\}_R+\delta(r)\alpha(s)-\alpha(r)\delta(s).
\end{align*}
Hence $\alpha$ is an even Poisson superderivation and $\delta$ is an even derivation satisfying equation \eqref{eq.delta_condition}.

Conversely, suppose that $\alpha$ is an even Poisson superderivation and that $\delta$ is an even superderivation satisfying equation \eqref{eq.delta_condition}. Define a bracket on $R[x]$ by
\begin{equation*}
\{rx^i,sx^j\}=(\{r,s\}_R-j\alpha(r)s+ir\alpha(s))x^{i+j}+(ir\delta(s)-j\delta(r)s)x^{i+j-1}
\end{equation*}
for all monomials $rx^i,sx^j$ in $R[x]$. Note this bracket satisfies equation \eqref{eq.Poisson_Ore_bracket} and $\{f,g\}=-(-1)^{|f||g|}\{g,f\}$ for homogeneous $f,g$. Also, $\{f,-\}$ is a superderivation of degree $|f|$ since $\alpha$ and $\delta$ are even superderivations. It remains to check the Jacobi identity: for monomials $rx^i,sx^j,tx^k\in R[x]$,
\begin{equation}\label{eq.Jacobi}
(-1)^{|r||t|}\{rx^i,\{sx^j,tx^k\}\}+(-1)^{|r||s|}\{sx^j,\{tx^k,rx^i\}\}+(-1)^{|s||t|}\{tx^k,\{rx^i,sx^j\}\}=0.
\end{equation}
We proceed by induction on $i,j,k$. The case $i=j=k=0$ is trivial, so suppose equation \eqref{eq.Jacobi} holds for some $i$ with $j=k=0$. Then
\begin{align*}
&(-1)^{|r||t|}\{rx^{i+1},\{s,t\}\}+(-1)^{|r||s|}\{s,\{t,rx^{i+1}\}\}+(-1)^{|s||t|}\{t,\{rx^{i+1},s\}\}\\
& \quad = \Big((-1)^{|r||t|}\{rx^i,\{s,t\}\}+(-1)^{|t||s|}\{s,\{t,rx^i\}\}+(-1)^{|s||t|}\{t,\{rx^i,s\}\}\Big)x\\
& \qquad + (-1)^{|r||t|}rx^i(\{x,\{s,t\}\}+\{s,\{t,x\}\}+(-1)^{|s||t|}\{t,\{x,s\}\})\\
& \quad = 0
\end{align*}
by the Leibniz rule and the induction hypothesis. One can similarly perform induction on $j$ for the case $k=0$, then perform induction on $k$ to complete the proof.
\end{proof}

The following lemma is easy by induction (c.f. \cite[Lemma 1.2]{Oh_PoissonOre}).

\begin{lemma}
\label{lem.ders_and_gens}
Let $R$ be a Poisson superalgebra generated as a superalgebra by a set $X$, and let $\alpha,\delta$ be even superderivations of $R$.
\begin{enumerate}
    \item If $\alpha(\{r,s\})=\{\alpha(r),s\}+\{r,\alpha(s)\}$ for $r,s\in X$, then $\alpha$ is a Poisson superderivation.
    \item If $\alpha,\delta$ satisfy equation \eqref{eq.delta_condition} for all elements in $X$, then $\alpha,\delta$ satisfy equation \eqref{eq.delta_condition} for all elements in $R$.
\end{enumerate}
\end{lemma}

For many Poisson superalgebras, verifying that the bracket is indeed a Poisson bracket requires an induction on several indices as in the proof above. Thus, as the following example demonstrates, Theorem \ref{thm.Poisson_Ore_extension} and Lemma \ref{lem.ders_and_gens} do most of the work for us.

\begin{example}\label{eg.Poisson_Ore}
Let $R'=\kk[x_1,x_2,\ldots,x_n\mid y_1,y_2\ldots,y_m]$ denote the polynomial superalgebra on even generators $x_i$ and odd generators $y_j$. Let $(\lambda_{ij})$ be an $n\times n$ skew-symmetric matrix, let $(\mu_{ij})$ be an $m\times m$ skew-symmetric matrix, and let $(\xi_{ij})$ be any $n\times m$ matrix. Define a bracket on $R'$ by
\begin{align*}
    \{x_i,x_j\}&:=\lambda_{ij}x_ix_j\\
    \{y_i,y_j\}&:=\mu_{ij}y_iy_j\\
    \{x_i,y_j\}&:=\xi_{ij}x_iy_j.
\end{align*}
We will use Theorem \ref{thm.Poisson_Ore_extension} to show this bracket gives $R'$ the structure of a Poisson superalgebra. Indeed, recall from Example \ref{eg.skew_symm} that the subsuperalgebra $P$ generated by the $y_j$ is a Poisson superalgebra under the induced bracket. Thus we only need to append the even variables $x_1,\ldots,x_n$ which can be done by constructing appropriate even Poisson superderivations.

Indeed, first note that if $V$ is a super vector space and $\phi$ is an element of the dual space $V^*$, then we can define an superderivation $\iota_\phi$ of degree $|\phi|$ on the exterior algebra by $\iota_\phi(v):=\phi(v)$ for $v\in\bigwedge\nolimits^1 V=V$. In particular, let $V$ be the super vector space generated by the $y_j$ and consider the linear maps $e_j:V\rightarrow \kk$ defined by $e_j(y_i):=\delta_{ij}$ for $1\leq j\leq m$, where $\delta_{ij}$ denotes the Kronecker delta; note that $\bigwedge V$ is the subsuperalgebra $P$ of $R'$. By the above discussion, the maps $e_j$ induce odd superderivations on $\bigwedge V$ which we denote by $\iota_j$, from which it follows that each $y_j\iota_j$ is an even superderivation of $\bigwedge V$. Similarly, one can easily verify that we can extend the $y_j\iota_j$ to even superderivations of $\kk[x_1,\ldots,x_r\mid y_1,\ldots,y_m]$ for any $1\leq r\leq n-1$ by recursively defining $\iota_j(x_\ell z):=x_\ell\iota_j(z)$ for $1\leq \ell\leq r$ and $z\in \kk[x_1,\ldots,x_r\mid y_1,\ldots,y_m]$. Abusing notation, we denote these superderivations by $y_j\iota_j$ as well.

Now, fix $0\leq r\leq n-1$ and define $\alpha_{r+1}:=\sum_{k=1}^r\lambda_{r+1,k}x_k\frac{\partial}{\partial x_k}+\sum_{\ell=1}^m\xi_{r+1,\ell}y_\ell\iota_\ell$. Note that $\alpha_{r+1}$ is an even superderivation of $\kk[x_1,\ldots,x_r\mid y_1,\ldots,y_m]$ by the above paragraph. Moreover, by using Lemma \ref{lem.ders_and_gens} one can easily verify that $\alpha_{r+1}$ is in fact an even Poisson superderivation. Further, one computes $\alpha_{r+1}(x_k)=\lambda_{r+1,k}x_k$ and $\alpha_{r+1}(y_{\ell})=\xi_{r+1,{\ell}}y_{\ell}$ for each $1\leq k\leq r$ and $1\leq\ell\leq m$. Hence by repeatedly applying Theorem \ref{thm.Poisson_Ore_extension} with the $\alpha_{r+1}$ constructed above and $\delta_{r+1}:=0$ for each $0\leq r\leq n-1$, we see that $R'$ is indeed a Poisson superalgebra under the given bracket.
\end{example}

Before studying the universal enveloping algebra of Poisson-Ore extensions, we briefly recall the definition of an Ore extension of a ring.

\begin{definition}
Let $S$ be a ring, let $\sigma:S\rightarrow S$ be a ring endomorphism, and let $\delta:S\rightarrow S$ be an Ore-$\sigma$-derivation of $S$; that is, $\delta$ is a homomorphism of abelian groups satisfying
\begin{equation*}
\delta(s_1s_2)=\sigma(s_1)\delta(s_2)+\delta(s_1)s_2
\end{equation*}
for $s_1,s_2\in S$. Then the \emph{Ore extension} $S[x;\sigma,\delta]$ is the free left $S$-module $S[x]$ equipped with the unique multiplication such that it obtains the structure of a ring, and such that
\begin{equation*} xs = \sigma(s)x + \delta(s)\end{equation*}
for all $s\in S$.
\end{definition}

\begin{remark}
From here on we will refer to Ore-$\sigma$-derivations as simply $\sigma$-derivations to match standard terminology. The reason we did not do this in the above definition is because an Ore-$\sigma$-derivation is not a $\sigma$-(super)derivation in the sense of Definition \ref{defn.superder}, in general.
\end{remark}

Let $R$ be a Poisson superalgebra and let $R'=R[x;\alpha,\delta]_p$ be a Poisson-Ore extension of $R$. We wish to show that $U(R')$ is an iterated Ore extension of $U(R)$. In particular, for appropriate $\sigma_1,\sigma_2,\eta_1,\eta_2$, we wish to show that $U(R')\cong U(R)[m_x;\sigma_1,\eta_1][h_x;\sigma_2,\eta_2]$. To determine what $\sigma_1,\eta_1$ should be, observe that in $U(R')$
\begin{equation*}
m_xm_r=m_rm_x \quad \text{and} \quad m_xh_r=h_rm_x-m_{\{r,x\}}=h_rm_x+m_{\alpha(r)}m_x+m_{\delta(r)}
\end{equation*}
for $r\in R$, so
\begin{align*}
\sigma_1(m_r)&:=m_r, & \sigma_1(h_r)&:=h_r+m_{\alpha(r)}, \\ \eta_1(m_r)&:=0, & \eta_1(h_r)&:=m_{\delta(r)}.
\end{align*}
Similarly, one has
\begin{align*}
\sigma_2(m_r)&:=m_r, & \sigma_2(h_r)&:=h_r+m_{\alpha(r)}, & \sigma_2(m_x)&:=m_x,\\
\eta_2(m_r)&:=m_{\alpha(r)}m_x+m_{\delta(r)}, &  \eta_2(h_r)&:=(h_{\alpha(r)}+m_{\alpha^2(r)})m_x+m_{\delta\alpha(r)}+h_{\delta(r)}, & \eta_2(m_x)&:=0.
\end{align*}
To extend to general elements we declare $\sigma_1,\sigma_2$ to be algebra homomorphisms, and declare $\eta_1$ (resp. $\eta_2$) to be an $\sigma_1$-derivation (resp. $\sigma_2$-derivation); that is, we recursively define $\eta_1(m_rz):=\sigma_1(m_r)\eta_1(z)+\eta_1(m_r)z$ with $\eta_1(h_rz)$ being defined similarly, and likewise for $\eta_2$.

\begin{lemma}
The maps $\sigma_1$ and $\sigma_2$ are well-defined algebra automorphisms of $U(R)$, and $\eta_1$ (resp. $\eta_2$) is a well-defined $\sigma_1$-derivation (resp. $\sigma_2$-derivation) of $U(R)$.
\end{lemma}
\begin{proof}
We first show that $\sigma_1$ is a well-defined automorphism. Indeed, first note that the triple $(U(R),f,g)$, where $f:R\rightarrow U(R)$ is defined by $f(r):=m_r$ and $g:R\rightarrow U(R)$ is defined by $g(r):=m_{\alpha(r)}+h_r$, satisfies property \textbf{P}. Hence, by the universal property of $U(R)$ there is a unique superalgebra homomorphism $\phi:U(R)\rightarrow U(R)$ such that the following  diagram commutes
\begin{equation*}
\xymatrixcolsep{4pc}\xymatrixrowsep{4pc}\xymatrix{
   R \ar@<2pt>[r]^{m} \ar@<-2pt>[r]_{h} \ar@<2pt>[rd]^{f} \ar@<-2pt>[rd]_{g} & U(R) \ar@{-->}[d]^{\phi} \\ & U(R)
}
\end{equation*}
This map $\phi$ is easily seen to be $\sigma_1$, and it is an algebra homomorphism since over a field any superalgebra (homomorphism) is an algebra (homomorphism). A similar argument can be used to show that $\sigma_1$ has a well-defined inverse, namely the map $m_r\rightarrow m_r$ and $h_r\rightarrow h_r-m_{\alpha(r)}$, so $\sigma_1$ is in fact an algebra automorphism. That $\sigma_2$ is a well-defined algebra automorphism follows immediately since $U(R)[m_x;\sigma_1,\eta_1]\cong\bigoplus_{n\geq 0}U(R) m_x^n$ as left $U(R)$-modules, and $\sigma_2|_{U(R)}=\sigma_1$.

To see that $\eta_1$ is well-defined, first notice that $\eta_1$ is well-defined on any fixed representative of any coset due to its recursive definition. Thus it suffices to prove $\eta_1$ annihilates any relation of $U(R)$. This is fairly simple and is left to the reader. That $\eta_1$ is in fact an $\sigma_1$-derivation is obvious.

In order to prove $\eta_2$ is well-defined, one first proves that $\eta_2$ also annihilates any relation of $U(R)$, proving it is well-defined on $U(R)$. Though this is more tedious than for $\eta_1$, it still is not difficult and is again left to the reader. That $\eta_2$ is well-defined then follows immediately from the aforementioned observation $U(R)[m_x;\sigma_1,\eta_1]\cong\bigoplus_{n\geq 0}U(R) m_x^n$ as left $U(R)$-modules. Finally, as with $\eta_1$, that $\eta_2$ is an $\sigma_2$-derivation is obvious.
\end{proof}

By construction, the obvious candidate for an isomorphism $\phi:U(R')\rightarrow U(R)[m_x;\sigma_1,\eta_1][h_x;\sigma_2,\eta_2]$ is
\begin{align}
\phi(m_r):=m_r, \quad \phi(h_r):=h_r, \quad
\phi(m_x):=m_x, \quad \phi(h_x):=h_x,
\end{align}
for $r\in R$. In order to show this is indeed an isomorphism however, we first need to show it is a well-defined homomorphism. Indeed, recall from the proof of Theorem \ref{thm.enveloping algebra} that $U(R')$ is the quotient by some two-sided ideal $I$ of the free superalgebra $F$ generated by $\{m_z,h_z\mid z\in R'\}$. Consider the superalgebra homomorphism $\Phi:F\rightarrow U(R)[m_x;\sigma_1,\eta_1][h_x;\sigma_2,\eta_2]$ defined on the generators of $F$ by
\begin{equation*}
\Phi(m_z) := \sum_{i=0}^\infty m_{c_i}x^i \qquad \text{and} \qquad  \Phi(h_z) := \sum_{i=0}^\infty(im_{c_i}m_x^{i-1}h_x+m_x^ih_{c_i}).
\end{equation*}
for $z=\sum_{i=0}^\infty c_ix^i\in R'=R[x]$. The following lemma shows that $\Phi(I)=0$, from which it follows the induced map $\phi:=\overline\Phi:U(R')\rightarrow U(R)[m_x;\sigma_1,\eta_1][h_x;\sigma_2,\eta_2]$ is a well-defined superalgebra homomorphism.

\begin{lemma}
\label{lemma.phi_well-def}
Let $\Phi:F\rightarrow U(R)[m_x;\sigma_1,\eta_1][h_x;\sigma_2,\eta_2]$ be as above. Then $\Phi(I)=0$.
\end{lemma}
\begin{proof}
We show $\Phi(m_{\{z,w\}}+(-1)^{|z||w|}m_wh_z-h_zm_w)=0$, leaving the computations for the other relations to the reader. Indeed, let $z=\sum_{i=0}^\infty c_ix^i, w=\sum_{j=0}^\infty d_jx^j$, and note $|z|=|c_i|, |w|=|d_j|$ for all $i,j$. Recall from the proof of Theorem \ref{thm.Poisson_Ore_extension} that the bracket on $R'$ is
\begin{equation*}
\{c_ix^i,d_jx^j\}=(\{c_i,d_j\}-(-1)^{|c_i||d_j|}jd_j\alpha(c_i)+ic_i\alpha(d_j))x^{i+j}+(ic_i\delta(d_j)-(-1)^{|c_i||d_j|}jd_j\delta(c_i))x^{i+j-1}.
\end{equation*}
Thus, using the equality $h_{c_i}m_x^j=m_x^jh_{c_i}-jm_{\alpha(c_i)}m_x^j-jm_{\delta(c_i)}m_x^{j-1}$ which can be proved via induction, we have
\begin{align*}
\Phi(m_{\{z,w\}}+&(-1)^{|z||w|}m_wh_z-h_zm_w)\\
&= \Phi\Big(m_{\sum_{i,j=0}^\infty\big[\big(\{c_i,d_j\}-(-1)^{|z||w|}jd_j\alpha(c_i)+ic_i\alpha(d_j)\big)x^{i+j}+(ic_i\delta(d_j)-\big(-1)^{|z||w|}jd_j\delta(c_i)\big)x^{i+j-1}\big]}\Big)\\
& \quad +(-1)^{|z||w|}\Phi(m_{\sum_{j=0}^\infty d_jx^j}h_{\sum_{i=0}^\infty c_ix^i})-\Phi(h_{\sum_{i=0}^\infty c_ix^i}m_{\sum_{j=0}^\infty d_jx^j})\\
&= \sum_{i,j=0}^\infty\Big[m_{\{c_i,d_j\}-(-1)^{|z||w|}jd_j\alpha(c_i)+ic_i\alpha(d_j)}m_x^{i+j}+m_{ic_i\delta(d_j)-(-1)^{|z||w|}jd_j\delta(c_i)}m_x^{i+j-1}\Big]\\
& \quad +(-1)^{|z||w|}\Big(\sum_{j=0}^\infty m_{d_j}m_x^j\Big)\Big(\sum_{i=0}^\infty (im_{c_i}m_x^{i-1}h_x+m_x^ih_{c_i})\Big)\\
& \quad -\Big(\sum_{i=0}^\infty (im_{c_i}m_x^{i-1}h_x+m_x^ih_{c_i})\Big)\Big(\sum_{j=0}^\infty m_{d_j}m_x^j\Big)\\
&= \sum_{i,j=0}^\infty\Big[m_{\{c_i,d_j\}}m_x^{i+j}-(-1)^{|z||w|}jm_{d_j\alpha(c_i)}m_x^{i+j}+im_{c_i\alpha(d_j)}m_x^{i+j}+im_{c_i\delta(d_j)}m_x^{i+j-1}\\
& \qquad -(-1)^{|z||w|}jm_{d_j\delta(c_i)}m_x^{i+j-1}+(-1)^{|z||w|}im_{d_jc_i}m_x^{i+j-1}h_x+(-1)^{|z||w|}m_{d_j}m_x^{i+j}h_{c_i}\\
& \qquad -im_{c_i}m_x^{i-1}h_xm_{d_j}m_x^j-m_x^ih_{c_i}m_{d_j}m_x^j\Big]\\
&= \sum_{i,j=0}^\infty\Big[m_{\{c_i,d_j\}}m_x^{i+j}-(-1)^{|z||w|}jm_{d_j\alpha(c_i)}m_x^{i+j}+im_{c_i\alpha(d_j)}m_x^{i+j}+im_{c_i\delta(d_j)}m_x^{i+j-1}\\
& \qquad -(-1)^{|z||w|}jm_{d_j\delta(c_i)}m_x^{i+j-1}+(-1)^{|z||w|}im_{d_jc_i}m_x^{i+j-1}h_x+(-1)^{|z||w|}m_{d_j}m_x^{i+j}h_{c_i}\\
& \qquad -im_{c_i}m_x^{i-1}\big(m_{d_j}h_x+m_{\alpha(d_j)}m_x+m_{\delta(d_j)}\big)m_x^j-m_x^i\big(m_{\{c_i,d_j\}}+(-1)^{|z||w|}m_{d_j}h_{c_i}\big)m_x^j\Big]\\
&= \sum_{i,j=0}^\infty\Big[m_{\{c_i,d_j\}}m_x^{i+j}-(-1)^{|z||w|}jm_{d_j\alpha(c_i)}m_x^{i+j}+im_{c_i\alpha(d_j)}m_x^{i+j}+im_{c_i\delta(d_j)}m_x^{i+j-1}\\
& \qquad -(-1)^{|z||w|}jm_{d_j\delta(c_i)}m_x^{i+j-1}+(-1)^{|z||w|}im_{d_jc_i}m_x^{i+j-1}h_x+(-1)^{|z||w|}m_{d_j}m_x^{i+j}h_{c_i}\\
& \qquad -im_{c_i}m_x^{i-1}\big(m_{d_j}h_x+m_{\alpha(d_j)}m_x+m_{\delta(d_j)}\big)m_x^j-m_x^im_{\{c_i,d_j\}}m_x^j\\
& \qquad -(-1)^{|z||w|}m_x^im_{d_j}\big(m_x^jh_{c_i}-jm_{\alpha(c_i)}m_x^j-jm_{\delta(c_i)}m_x^{j-1}\big)\Big]\\
& =0.
\end{align*}
\end{proof}

\begin{theorem}
\label{thm.Poisson-Ore_UEA}
The superalgebra homomorphism $\phi:U(R')\rightarrow U(R)[m_x;\sigma_1,\eta_1][h_x;\sigma_2,\eta_2]$ is an isomorphism.
\end{theorem}
\begin{proof}
By Lemma \ref{lemma.phi_well-def}, it suffices to show $\phi$ is bijective; we will construct the inverse map $\psi$. Indeed, consider the triple $(U(R'),f,g)$, where $f,g:R\rightarrow U(R')$ are defined by $f(r):=m_r$ and $g(r):=h_r$, respectively. One easily shows this triple satisfies property \textbf{P} with respect to $R$, so there is a unique superalgebra homomorphism $\theta:U(R)\rightarrow U(R')$ such that the following diagram commutes:
\begin{equation*}
\xymatrixcolsep{4pc}\xymatrixrowsep{4pc}\xymatrix{
   R \ar@<2pt>[r]^{m} \ar@<-2pt>[r]_{h} \ar@<2pt>[rd]^{f} \ar@<-2pt>[rd]_{g} & U(R) \ar@{-->}[d]^{\theta} \\ & U(R')
}
\end{equation*}
Now, define $\psi(wm_x^ih_x^j):=\theta(w)m_x^ih_x^j$ for $w\in U(R)$. It follows easily from the definitions of $\phi$ and $\psi$ that they are inverse functions, and thus an inverse pair of superalgebra isomorphisms.
\end{proof}

\begin{corollary}
\label{cor.Ore_properties}
Let $R$ be a Poisson superalgebra and let $R'$ be an iterated Poisson-Ore extension of $R$. Then $U(R')$ inherits the following properties from $U(R)$:
\begin{enumerate}
    \item being a domain;
    \item being Noetherian;
    \item having finite global dimension;
    \item having finite Krull dimension;
    \item being twisted Calabi-Yau.
\end{enumerate}
\end{corollary}
\begin{proof}
It is well known that properties (1)-(4) are preserved by Ore extensions, see \cite{MR_NoetherianRings} for example, while the preservation of property (5) is proved in \cite{LWW_Calabi_Yau}.
\end{proof}

\begin{example}
Recall the Poisson superalgebra $R'$ with underlying superalgebra $\kk[x_1,\ldots,x_n\mid y_1,\ldots,y_m]$ discussed in Example \ref{eg.Poisson_Ore}, and consider the Poisson subsuperalgebra $P$ generated by the $y_j$. We proved in Example \ref{eg.Poisson_Ore} that $R'$ is an iterated Poisson-Ore extension of $P$. Hence by Theorem \ref{thm.Poisson-Ore_UEA} and Example \ref{eg.skew_symm_enveloping_alg}, $U(R')$ is an iterated Ore extension of
\begin{equation*}
    U(P)\cong\frac{\kk<m_{\theta_1},\ldots,m_{\theta_n},h_{\theta_1},\ldots,h_{\theta_n}>}{([m_{\theta_k},m_{\theta_\ell}]_{\gr},\, [h_{\theta_k},m_{\theta_\ell}]_{\gr}-\lambda_{k\ell}m_{\theta_\ell}m_{\theta_k},\, [h_{\theta_k},h_{\theta_\ell}]_{\gr}-\lambda_{k\ell}(m_{\theta_\ell}h_{\theta_k}-m_{\theta_k}h_{\theta_\ell}))}.
\end{equation*}
\end{example}

\subsection*{Data Availability Statement}
Data sharing is not applicable to this article as no new data were created or analyzed in this study.

\bibliographystyle{plain}
\bibliography{biblio}

\end{document}